\documentclass[12pt]{article}

\usepackage{amsfonts}
\usepackage{amsthm}
\usepackage{amsmath}
\usepackage{graphicx,epstopdf}
\usepackage{subfig}
\usepackage{array}
\usepackage{multirow}
\usepackage{lipsum}
\usepackage{amsfonts}
\usepackage{amssymb}
\usepackage{graphicx}
\usepackage{epstopdf}
\usepackage{algorithmic}
\usepackage{xcolor}
\usepackage{amsopn}
\usepackage{float}

\newtheorem{theorem}{Theorem}
\newtheorem{definition}{Definition}

\def\A{\mathcal{A}}
\def\prob{\mathbb P} 
\def\N{\mathcal N}

\def\R{\mathcal R}

\def\S{\mathcal S}
\def\I{\mathcal I}
\def\X{\mathcal X}
\def\U{\mathcal U}
\def\V{\mathcal V}
\def\E{\mathbb E}
\def\x{\mathbf x}
\def\y{\mathbf y}
\def\tB{\tilde B}
\newcommand*{\email}[1]{\texttt{#1}}
\title{\vspace{-1in}Global Density Analysis for an Off-Lattice Agent-Based Model
}

\date{\today}
\author{Michael A. Yereniuk \thanks{Department of Mathematical Sciences, Worcester Polytechnic Institute, Worcester, MA 
  (\email{mayereniuk@wpi.edu})}
\and Sarah D. Olson \thanks{Department of Mathematical Sciences, Worcester Polytechnic Institute, Worcester, MA 
  (\email{sdolson@wpi.edu})}}

\newfloat{glossary}{thp}{lop}
\floatname{glossary}{Glossary}

\begin{document}
\maketitle
\begin{abstract}
Agent-Based (AB) models are rule based and are a relatively simple discrete method that can be used to simulate complex interactions of many agents. The relative ease of implementing the computational model is often counter-balanced by the difficulty of performing rigorous analysis to determine emergent behaviors. In addition, without precise definitions of agent interactions, calculating existence of fixed points and their stability is not tractable from an analytical perspective and can become computationally expensive, involving potentially thousands of simulations. Through developing a precise definition of an off-lattice AB model with a specified interaction neighborhood, we develop a general method to determine a Global Recurrence Rule (GRR). This allows estimates of the state densities in time, which can be easily calculated for a range of parameters in the model. The utility of this framework is tested on an Epidemiological Agent-Based, E-AB, model where agents correspond to people that are in the susceptible, infected, or recovered states. The interaction neighborhoods of agents are determined in a mathematical formulation that allows the GRR to accurately predict the long term behavior and steady states. The modeling framework outlined will be generally applicable to many areas and can be easily extended.
\end{abstract}

% REQUIRED
\begin{center}
\end{center}

%%%%%%%%%%%%%%%%%%%%%%%%%%%%%%%%%%%%%%%%%%%%%%%%%%%%%%%%%%%%%%%%%%%%%%%%%%%%%
%%%%%%%%%%%%%%%%%%%%%%%%%%%%%%%%%%%%%%%%%%%%%%%%%%%%%%%%%%%%%%%%%%%%%%%%%%%%%
%%%%%%%%%%%%%%%%%%%%%%%%%%%%%%%%%%%%%%%%%%%%%%%%%%%%%%%%%%%%%%%%%%%%%%%%%%%%%
\section{Introduction}
When the goal is to understand a complex system of interacting agents that are decision makers, there are several different modeling frameworks from which to choose. Often, if the focus is on understanding and capturing each of the interactions and movement, an individual-based approach such as an Agent-Based (AB) model is used \cite{Hinkelmann2011,Holcombe,Laubenbacher2012, Packard1985}. When we are interested in global dynamics, or are more interested in how the field is evolving, density-based approaches are utilized, e.g. difference equations or systems of differential equations. For each of these frameworks, there are different pros and cons with respect to the ability to formalize and analyze a model, as well as the ease with which one can simulate the model \cite{North2014}. There are many challenges, which can arise due to noise, nonlinearities, and other spatial or temporal variations in the system \cite{Roberts2015}.

In AB models, the agents are each individually assessing their surrounding environment, potentially moving or changing state at each time increment based on a given set of rules \cite{Bonabeau7280}. The state-dependent rules could be deterministic or stochastic, and are quite often a nonlinear function based on information (e.g. other agents, states, or other environmental factors) in a locally defined interaction neighborhood \cite{Deutsch}. The movement can be on-lattice with discretely defined locations that are assigned a given probability. For example, if a lattice is a regular, two-dimensional grid, the on-lattice movement of  a given agent could be either horizontal or vertical movement to an adjacent lattice node at each time increment \cite{Goldsztein, Stevens}. 
Movement can also be off-lattice or continuous in space, where new locations can be determined {via specified rules or determined by solving systems of differential equations.}
Often, questions of interest concern the emergent behavior of a large number of interacting agents, which can be hard to capture at the continuous scale \cite{Holcombe}. We note that since this modeling framework is quite general, the agent could represent any feature of interest in a given system \cite{Pogson2006}, which is why these types of models are frequently used for social, biological, financial, and military applications \cite{An2017,Bonabeau7280,Chaturapruek,Cosgrove,Devitt-Lee,Goldsztein, Interian, Othmer,Stevens}. In terms of biological applications at the cellular level, AB models have been used to investigate  tumor growth where the agents are the individual cells that make up the tumor \cite{Interian}, sperm cell motility where the sperm are the individual agents \cite{Burkitt,Burkitt12}, and signaling pathways within and on the membrane of cells where agents are molecules and receptors \cite{Bonabeau7280}.

It is well known that as the number of agents in a system increases, this may cause simulations investigating long term dynamics to become intractable \cite{Holcombe}. Additionally, there is generally a desire to understand how model outcomes change with respect to varying parameter values \cite{Chen}, which again would necessitate possibly thousands of simulations. 
As described previously, there is not a universal or agreed upon standard to specify these models and, in many cases, the mathematical description of the rules is also not specified \cite{Hinkelmann2011}. 

The focus of this current work is an introduction of a theoretical formalism and subsequent analysis of an off-lattice AB model where agents exhibit stochastic behavior when moving and changing states.  {Although we focus on the case of off-lattice AB models, we note that this modeling framework has several similarities to cellular automata (CA) \cite{Deutsch, Hinkelmann2011} and dynamic network models \cite{Starnini12,Gross09}.  CA models also have agents but they are generally fixed in space and state changes of an agent are generally determined based on the state of the same neighboring agents. This is in contrast to our model where agents will move and will have proximity to different agents in time, having different interaction neighborhoods. In the case of dynamic network models, agents are generally at particular nodes and interact with other agents at their node or potentially other nodes that are connected with an edge. The dynamic portion could correspond to the dynamic movement of the agents on the network or it could correspond to the creation or deletion of certain edges or nodes in time. Since the AB models we consider have different agents interacting in time, there are similarities. However, in our case, movement and interaction neighborhoods are based on the same specified spatial scale whereas interactions in a dynamic network might occur over a range of spatial scales corresponding to the range of edge lengths. }

The analysis of the off-lattice AB model will rely heavily on a precise definition of the interaction neighborhoods of agents.  In contrast to other, more traditional approaches \cite{Deutsch}, we will view the interaction neighborhood as a region where an agent potentially exerts state changes to other agents. Specifically, the necessary notation for the AB model is outlined in Section \ref{notation}, which is similar to previous work on CA models \cite{Deutsch,Hinkelmann2011}. 
In Section \ref{grr-basic}, we detail how to derive a Global Recurrence Rule (GRR) to determine the expected value for the number of agents in each state when assuming that an agent's state and movement are solely determined by the agent's current status. To show the applicability of this formalism, in Section \ref{eca-results}, we illustrate how a GRR can be derived for an epidemiological-AB (E-AB) model that captures the spread of an infection such as influenza. The long term behavior and steady state solutions obtained for the infected, recovered, and susceptible states using our GRR have good agreement with simulations and we are able to prove stability of fixed points. In addition, we illustrate with the E-AB how to use additional information about the dynamics to develop a more refined local approximation of the neighborhoods, with reduced error. In Section \ref{eca}, we compare the different models and emphasize which assumptions need to be satisfied in order for the GRR to be a valid approximation for the E-AB model. 

%%%%%%%%%%%%%%%%%%%%%%%%%%%%%%%%%%%%%%%%%%%%%%%%%%%%%%%%%%%%%%%%%%%%%%%%%%%%%
%%%%%%%%%%%%%%%%%%%%%%%%%%%%%%%%%%%%%%%%%%%%%%%%%%%%%%%%%%%%%%%%%%%%%%%%%%%%%
%%%%%%%%%%%%%%%%%%%%%%%%%%%%%%%%%%%%%%%%%%%%%%%%%%%%%%%%%%%%%%%%%%%%%%%%%%%%%
%%%%%%%%%%%%%%%%%%%%%%%%%%%%%%%%%%%%%%%%%%%%%%%%%%%%%%%%%%%%%%%%%%%%%%%%%%%%%
\section{Initial definitions}\label{notationsec}
\subsection{AB notation}\label{notation}
We first need to create a precise definition of the properties of the agents and their interactions in order to determine the correct governing equations and hence be able to mathematically analyze the model.
We define a bounded region of interest $\Omega$ in which we track the agents.   
We suppose that we have a finite collection, $\X$, of $N$ agents indexed by $1,...,k-1,k,k+1,...,N$.
At every time $t\geq 0$, each agent $k$ is assigned one and only one state $s_k^t$ and location $\x_k^t$. 
We define the set containing all possible agent states as the state space $\Sigma = \{\U_1, \U_2, ... \}$. That is, for each agent $k$ and time $t$, we necessitate that $s_t^k = \U_i$ for exactly one state $\U_i \in \Sigma$.

The domain, $\Omega$, may be either a graph or network consisting of a discrete set of nodes connected by edges or it may be a continuous, bounded subset of $\mathbb{R}^M$ for some $M\in \mathbb{N}$.
If $\Omega$ is discrete, we say these agents exist on-lattice{, where they occur on nodes and travel along the graph edges}.  
Otherwise, if $\Omega$ is continuous, we say these agents exist off-lattice.  In either case, the agents can either be stationary or move.  
{The model keeps track of individual agents performing a deterministic or random walk, governed by a model specific probability distribution over a bounded domain} \cite{Chaturapruek,Davis1990,Othmer}. 
It is important to note that the scalings and distributions may be spatially, temporally, or state dependent. 

An individual-based model is defined by the pair $\A = (f,\N)$, where we denote the collection of neighborhoods for each agent as 
$\N=\left\{\N^1,\N^2,...,\N^N\right\}$ 
and $f$ is a local transition rule, which defines how agents change states \cite{Deutsch}.  One can visualize the agents and neighborhoods in the insert of Figure \ref{fig: insert}.
If the model specifies that each neighborhood is temporally or state dependent, we define $\N_t^j$ as the neighborhood of agent $j$ in state $s_t^j$ at iteration $t$. 
Since the agents move in time, these neighborhoods are time dependent.
We emphasize that the neighborhood $\N_t^j$ is the region of influence in which agent $j$ may exert state changes to other agents.
Traditional CA definitions define the neighborhood $\N_t^j$ as the region in which other agents exert state changes to agent $j$, where density analysis has been studied previously \cite{Deutsch}.  However, since we want to focus on locations of moving individuals that can influence state changes to others in a region, we assert the opposite---$\N_t^j$ is the region in which agent $j$ exerts state changes to other agents.  This perspective allows us greater freedom to model more realistic state and property-dependent neighborhoods.
The local transition rule is a function $f:\X \to \Sigma$.  Since each agent belongs to one and only one state, the local transition rule $f$ is well-defined.  The function assignment $s_{t+1}^k = f(s_t^k)$ depends conditionally on the neighborhoods in which $\x_t^k$ is contained. 

Define $B_t^{\V,\U}$ as the $\V\to \U$ transition region
at time $t$.  That is, $B_t^{\V,\U}$ is the region such that if, at time $t$, there is some agent $k$ such that $s_t^k = \V$ and $\x_t^k \in B_t^{\V,\U}$, then agent $k$ can transition to state $\U$ at time $t+1$.  In terms of our neighborhoods of influence, we can formally define the transition region as follows. 
\begin{definition}
The \underline{$\V\to \U$ transition region} at time $t$ is 
$B_t^{\V,\U} = \bigcup\limits_{j\in \mathcal{C}} \N^j$,
with indexing set 
$\mathcal{C} = \left\{
j \, |\, \exists k \text{ such that } s_t^k = \V \text{ and } \x_t^k \in \N_t^j \Rightarrow \prob\left(f(s_t^k)=\U\right)>0
\right\}.
$
\end{definition}

We define the transition region in this way, since, in general, agents in states other than $\U$ can cause agents to transition to state $\U$.
Our explicit definition of the transition region $B_t^{\V,\U}$ for each $\V, \U\in \Sigma$ will allow us to clearly define $f(s_t^k)$.  

We are interested in finding a GRR, which calculates the expected density of agents in a state at each iteration throughout $\Omega$.  
Let $\U \in \Sigma$, with $U_t$ denoting the number of agents in state $\U$ at iteration $t$ (that is, $U_t=|\{k: s_t^k=\U\}|$).    We denote $W(\V \to \U)$ to be the transition probability that an agent in state $\V$ at iteration $t$ transitions to state $\U$ at time $t+1$.  
We summarize our AB notation in Glossary \ref{glossary}.
\begin{glossary}[H]
\textit{
\begin{itemize}
\item $\X$: the collection of agents (indexed as $1,...,k-1,k,k+1,...,N$)
\item $\Sigma$: state space
\item $s_t^k$: the state of agent $k$ at time $t$
\item $\Omega$: the bounded region of interest
\item $\x_t^k$: the location of agent $k$ at time $t$
\item $\N_t^k$: neighborhood of agent $k$ at time $t$
\item $f:\X \to \Sigma$: the transition rule that assigns each agent at time $t$ to a\\ unique state at time $t+1$
\item $U_t$: the number of agents in state $\U$ at time $t$
\item $B_t^{\V,\U}$: the $\V \to \U$ transition region
\item $W(\V\to \U)$: the probability an agent in state $\V$ transitions to state $\U$
\end{itemize}}
\caption{Agent-Based Model Terms}
\label{glossary}
\end{glossary}

%%%%%%%%%%%%%%%%%%%%%%%%%%%%%%%%%%%%%%%%%%%%%%%%%%%%%%%%%%%%%%%%%%%%%%%%%%%%%
\subsection{Global recurrence rule}\label{grr-basic}
We now have the necessary notation to derive the expected density of 
{agents in}
a state at any given time.
The state of an agent $k$ at $t+1$ only depends on its state ($s_t^k$) and location ($\x_t^k$) at time $t$, making this a Markovian process \cite{Markov}.  Hence, the probability that agent $k$ is in state $\U$ at time $t+1$ given that the agent was in state $\V$ at time $t$ reduces to
\begin{equation}
\prob\left(s_{t+1}^{k} = \U \Big| s_t^k = \V\right) = \prob\left(\x_t^k \in B_t^{\V,\U}\right) W(\V \to \U),
\end{equation}
the product of the probability that agent $k$ is located in a $\V \to \U$ transition region with the probability that agent $k$ transitions from state $\V$ to state $\U$.
We can then find $\E(U_{t+1})$, the expected number of agents in state $\U$ at iteration $t+1$, by
\begin{align}
\E(U_{t+1})&=\E\left(\big|\{k:s_{t+1}^{k}= \U\} \big|\right)\\
&= \sum_{k=1}^N \prob\left(s_{t+1}^k = \U\right) 
\label{eq:sum_cells}
\\
&= \sum_{\V \in \Sigma} \sum_{\{k:s_t^{k}= \V\}} \prob\left(s_{t+1}^k= \U \Big|s_t^k = \V\right) 
\label{eq:sum_states}
\\
&= \sum_{\V \in \Sigma} \sum_{\{k:s_t^{k}= \V\}} \prob\left(\x_t^k \in B_t^{\V,\U}\right)W(\V \to \U) .
\label{GeneralGRR}
\end{align}
Note that equality between \eqref{eq:sum_cells} and \eqref{eq:sum_states} is due to the fact that we can partition the collection of agents $\X$ by the distinct states in $\Sigma$. 
This leads us to the definition of the GRR.
\begin{definition}
Let $\U,\V \in \Sigma, U_t = |\{k:s_t^k=\U\}|$, and $B_t^{\V,\U}$ be the $\V\to\U$ transition region at time $t$.  We define the \underline{GRR} as
\[
\E(U_{t+1}) = \sum_{\V \in \Sigma} \sum_{\{k:s_t^{k}= \V\}} \prob\left(\x_t^k \in B_t^{\V,\U}\right)W(\V \to \U). \\
\]
\label{defGRR}
\end{definition}

Thus, to find expected values of 
{the number of agents in each state} 
analytically, one just needs a framework to calculate both the probability of being in a transition region 
{as well as } 
the probability that an agent in the transition neighborhood will transition to a particular state.

%%%%%%%%%%%%%%%%%%%%%%%%%%%%%%%%%%%%%%%%%%%%%%%%%%%%%%%%%%%%%%%%%%%%%%%%%%%%%
%%%%%%%%%%%%%%%%%%%%%%%%%%%%%%%%%%%%%%%%%%%%%%%%%%%%%%%%%%%%%%%%%%%%%%%%%%%%%
%%%%%%%%%%%%%%%%%%%%%%%%%%%%%%%%%%%%%%%%%%%%%%%%%%%%%%%%%%%%%%%%%%%%%%%%%%%%%
%%%%%%%%%%%%%%%%%%%%%%%%%%%%%%%%%%%%%%%%%%%%%%%%%%%%%%%%%%%%%%%%%%%%%%%%%%%%%
\section{Application to disease dynamics}\label{eca-results}
%%%%%%%%%%%%%%%%%%%%%%%%%%%%%%%%%%%%%%%%%%%%%%%%%%%%%%%%%%%%%%%%%%%%%%%%%%%%%
\subsection{Phenomenological perspective}
Disease dynamics provides an interesting case study to determine the validity of the GRR.  Assume there are infected individuals in a population.  For simplicity, we can divide the remaining population into two classes, those who are susceptible to infection and those who were infected but cannot currently infect other individuals.  We denote these classes ``susceptible'' and ``recovered,'' respectively.   Further, suppose that after a finite time the recovered can become susceptible to infection again.  That is, an individual in the recovered state is temporarily conferred immunity before returning to the susceptible state. This is often referred to as a Susceptible-Infected-Recovered (SIR) Epidemiological model, where simulations and analysis have been an active research area for many years \cite{Holko16,Prieto2016,Volz09}, especially with respect to endemic equilibrium sizes \cite{Ma, Miller,Volz07} and infectivity wave speed \cite{Fuentes99,Marziano}. 
The modeling framework for SIR Epidemiological studies has been based on differential equations, networks, and AB models; each approach has provided some successes. CA models (on-lattice) with no movement have been studied and compared to differential equation models, extending our understanding of disease spread \cite{Fuentes99}. Extensions of CA models to real-world data in a geographic region with movement of people \cite{Holko16} gives more realistic disease spread where intervention strategies can be tested, but analysis of the model is often intractable. Challenges still exist to capture the relevant dynamics and to make time sensitive and accurate predictions with regards to disease outbreaks \cite{BALL201563,Lloyd,Ma,Marziano,Miller,Prieto2016,Roberts2015}.

In terms of the AB framework presented in Section \ref{notationsec}, it is relatively straightforward to implement an E-AB model.  There are only three states: $\S$ (susceptible), $\I$ (infected), and $\R$ (recovered).  In addition, the only neighborhoods of interest are those belonging to the infected agents since they will influence the state change of susceptible agents in their region of influence. 

%%%%%%%%%%%%%%%%%%%%%%%%%%%%%%%%%%%%%%%%%%%%%%%%%%%%%%%%%%%%%%%%%%%%%%%%%%%%%
\subsection{E-AB model}
To simplify, we let the continuous domain, 
$\Omega$, 
of the E-AB be the 
unit square.
The agents remain in the infected and recovered states for $T_i$ and $T_R$ iterations, respectively.
Thus, our state space for the E-AB is $\Sigma = \{\S,\I_1,\I_2,...,\I_{T_I},\R_1,R_2,...,R_{T_R}\}$.  This dynamic is also referred to as an $SI^{T_I}R^{T_R}$ model \cite{Ma}.

We initialize $N$ agents in $\Omega$ such that $N-1$ agents are in state $\S$ ($S_0 = N-1$) and one agent is in state $\I_1$ ($I_0=1$), 
where $S_t=|\{k:s_t^k=\S\}|$ and 
$I_t = \cup_{j=1}^{T_I} |\{k:s_t^k= \I_j\}|$ for each time $t$.
We index the initially infected agent as $k=1$ and initialize its location in the center of the region $\Omega$. 
The susceptible agents are randomly initialized following a uniform random distribution (i.e. $\x_0^k \sim Uniform(\Omega)$ for $k=2,3,\dots,N$). 

Each agent\footnote{
For simplicity, every agent in this model moves according to the same rules.  However, one could produce a model where each state moves differently.  For example, the infected agents could move at a much smaller spatial step than the susceptible or recovered agents.}
moves by a uniform random walk inside $\Omega$.  
If $\x_t^k = (x,y)$, then $\x_{t+1}^k = (x+\Delta r \cos\theta, y+\Delta r \sin\theta)$,
where $\theta \sim Uniform[0,2\pi)$ and $\Delta r \ll 1$ is the constant radial spatial step.  
Reflective boundary conditions are enforced along $\partial \Omega$.  
That is, if an agent hits the boundary (or is about to move outside of $\Omega$), it is shifted a distance $\Delta r$ into $\Omega$ along the direction normal to the boundary. 

For our E-AB, we assume that the infectivity neighborhood of any infected agent $k$ is radially symmetric with radius $\rho_0$.  That is, 
\\
$\N_t^k = \left\{\y \in \Omega: ||\y-\x_t^k||_2\leq \rho_0\right\}$,
the collection of all points of a distance less than $\rho_0$ away from agent $k$, is
the area in which susceptible agents can become infected by agent $k$.  

Now consider any agent $k$ such that $s_t^k = \S$.  In order for agent $k$ to become infected, we require $\x_t^k$ to be in an infected neighborhood, regardless of the iteration of infectivity.  We define the $\S \to \I_1$ transition region as 
$B_t^{\S,\I_1} = \bigcup_{\{k: s_t^k = \I_j, \exists j\}}\N_t^k$.  
Recall that the $\S$ to $\I_1$ transition region is the region in which an agent in state $\S$ can transition to state $\I_1$. The susceptible agent has the potential to become infected when in at least one neighborhood of an infected agent at any state of the infection (for $j=1,\ldots,T_I$).
In this simple E-AB model, the number of infectivity neighborhoods in which agent $k$ is located does not affect the probability of agent $k$ being infected.  
The susceptible agents located inside $B_t^{\S,\I_1}$ become infected with probability $1-\kappa$, where $\kappa\in [0,1]$ is the contact tolerance.  
For simplification, we will assume that $\rho_0$ and $\kappa$ are scalar constants\footnote{
Our E-AB is a toy example to demonstrate the efficacy of the GRR.  For simplification, $\rho_0$ and $\kappa$ are constants.  In practice, $\rho_0$ and $\kappa$ should be random variables drawn from specific probability distributions, such as the models found in \cite{Lloyd}.} and the transition rules between states are given below.

\begin{definition}The \underline{local transition rule $f:\X\to \Sigma$}, such that $s_{t+1}^k = f(s_t^k)$ are given as follows:
{\small
%\vspace{-1in}
\begin{align}
&\text{If } s_t^k = \S: 
\quad
f(s_t^k) = \begin{cases}
\I_1&: \text{ if }\x_t^k \in B_t^{\S,\I_1}\text{ and }\kappa>X,\text{ where }X\sim Uniform[0,1],\\
\S&: \text{otherwise,}
\end{cases}
\label{S_TransitionRule}
\\
&\text{If }s_t^k = \I_j, \, \text{{for some} } j=1,2,...,T_I:
\quad 
f(s_k^t) = 
\begin{cases}
\I_{j+1}&: \text{ if }1\leq j < T_I ,\\
\R&: \text{ if }j=T_I,
\end{cases}
\label{I_TransitionRule}
\\
&\text{If } s_t^k = \R_m, \, \text{{for some} } m=1,2,...,T_R:
\quad
f(s_t^k) = 
\begin{cases}
\R_{m+1}&: \text{ if }1\leq m < T_R, \\
\S&: \text{ if }m=T_R.
\end{cases}.
\label{R_TransitionRule}
\end{align}
}
\end{definition}

Figure \ref{fig: Simulation} illustrates the off-lattice E-AB simulation as outlined above using $N=10000$ agents, where the susceptible, infected, and recovered agents are colored as black, red, and {blue}, respectively.  We implemented each iteration by first determining the region of infectivity from a constant infectivity radius of $\rho_0=0.04$.  We then updated the agent states according to the above transition rules \eqref{S_TransitionRule}--\eqref{R_TransitionRule} with contact tolerance $\kappa=0.95$.  This ``high'' contact tolerance relates to a ``low'' probability of a susceptible agent becoming infected.  Moreover, the time spent in the infective state and the time spent in the recovered state are $T_I = T_R = 30$.
Finally, we update the agent location by performing unbiased random walks with $\Delta r = 0.001$.  
\begin{figure}
   \centering
   \subfloat[$t=10$]{ \includegraphics[width=0.38\textwidth, height=0.3\textwidth]{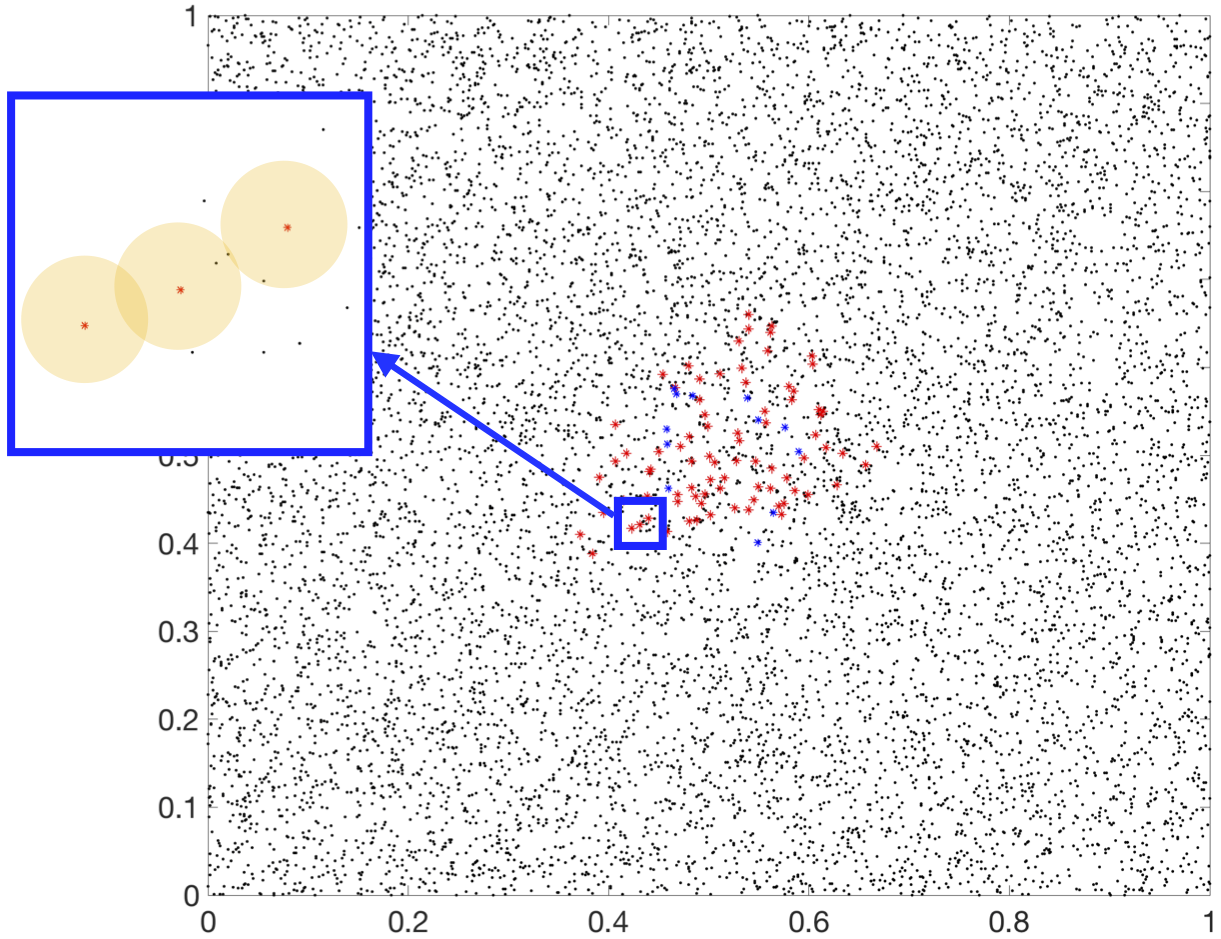} \label{fig: insert}}
   \qquad
   \subfloat[$t=20$]{ \includegraphics[width=0.35\textwidth, height=0.3\textwidth]{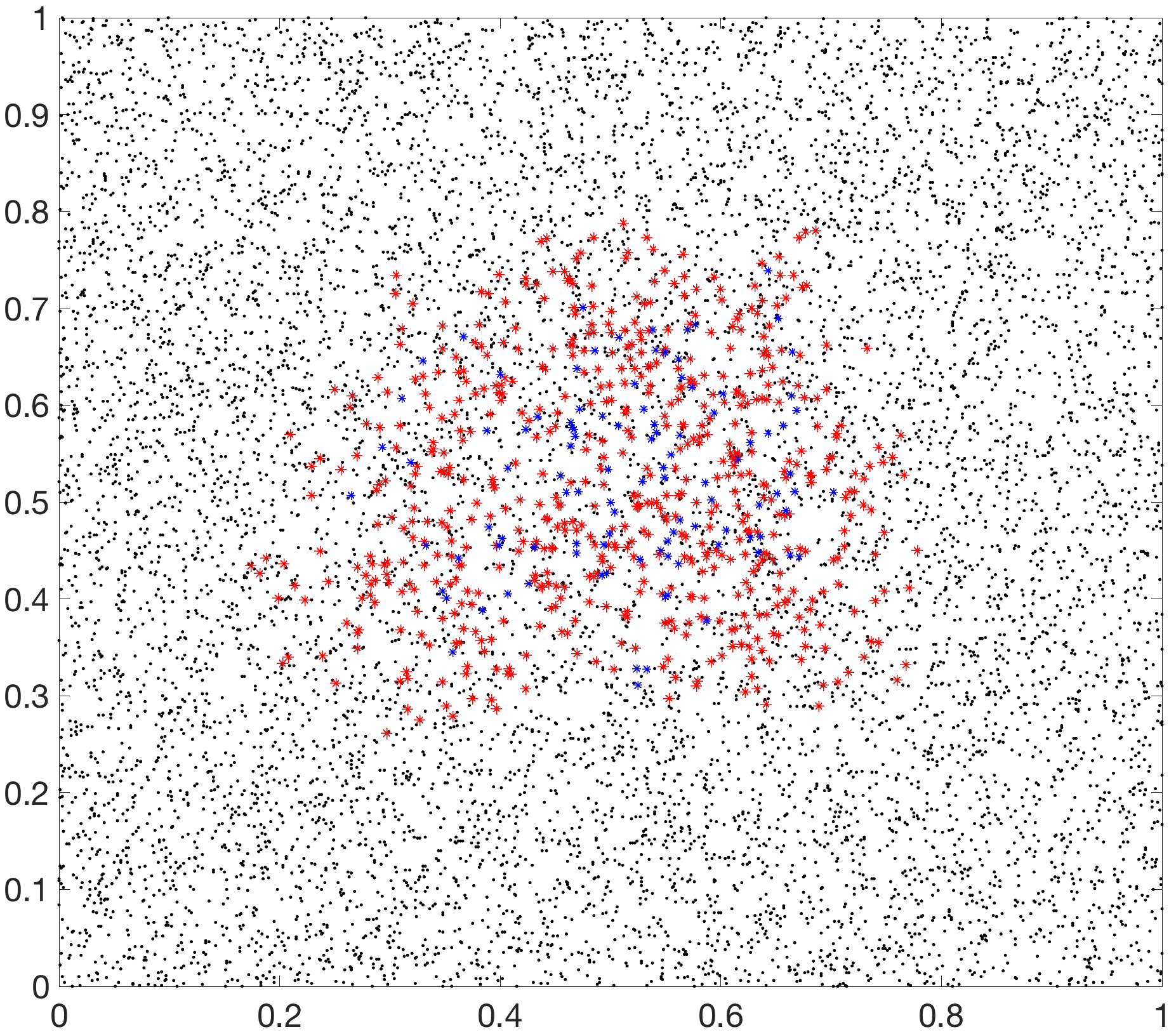} }\\
   \subfloat[$t=30$]{ \includegraphics[width=0.35\textwidth, height=0.3\textwidth]{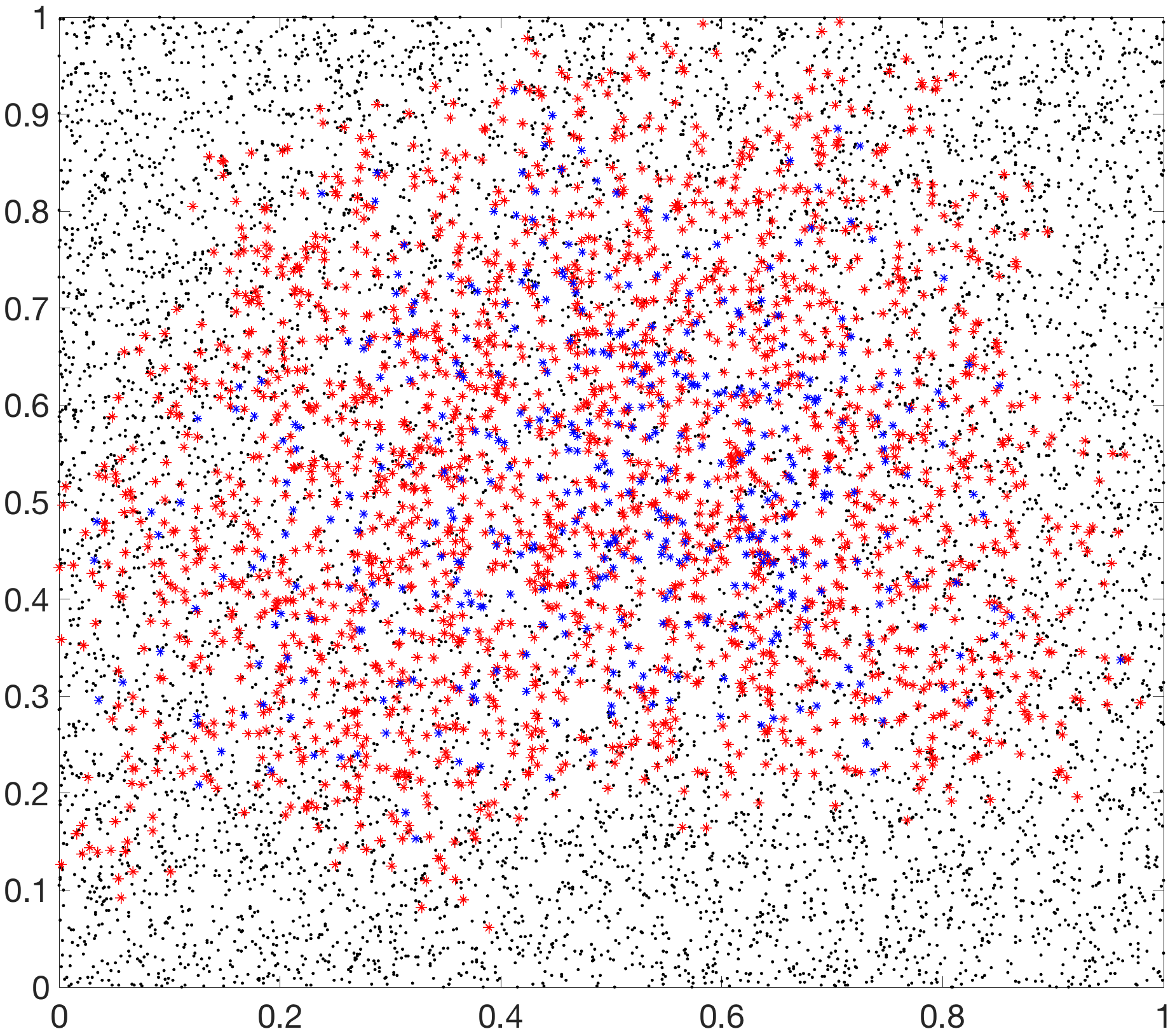}}
   \qquad
   \subfloat[$t=40$]{ \includegraphics[width=0.35\textwidth, height=0.3\textwidth]{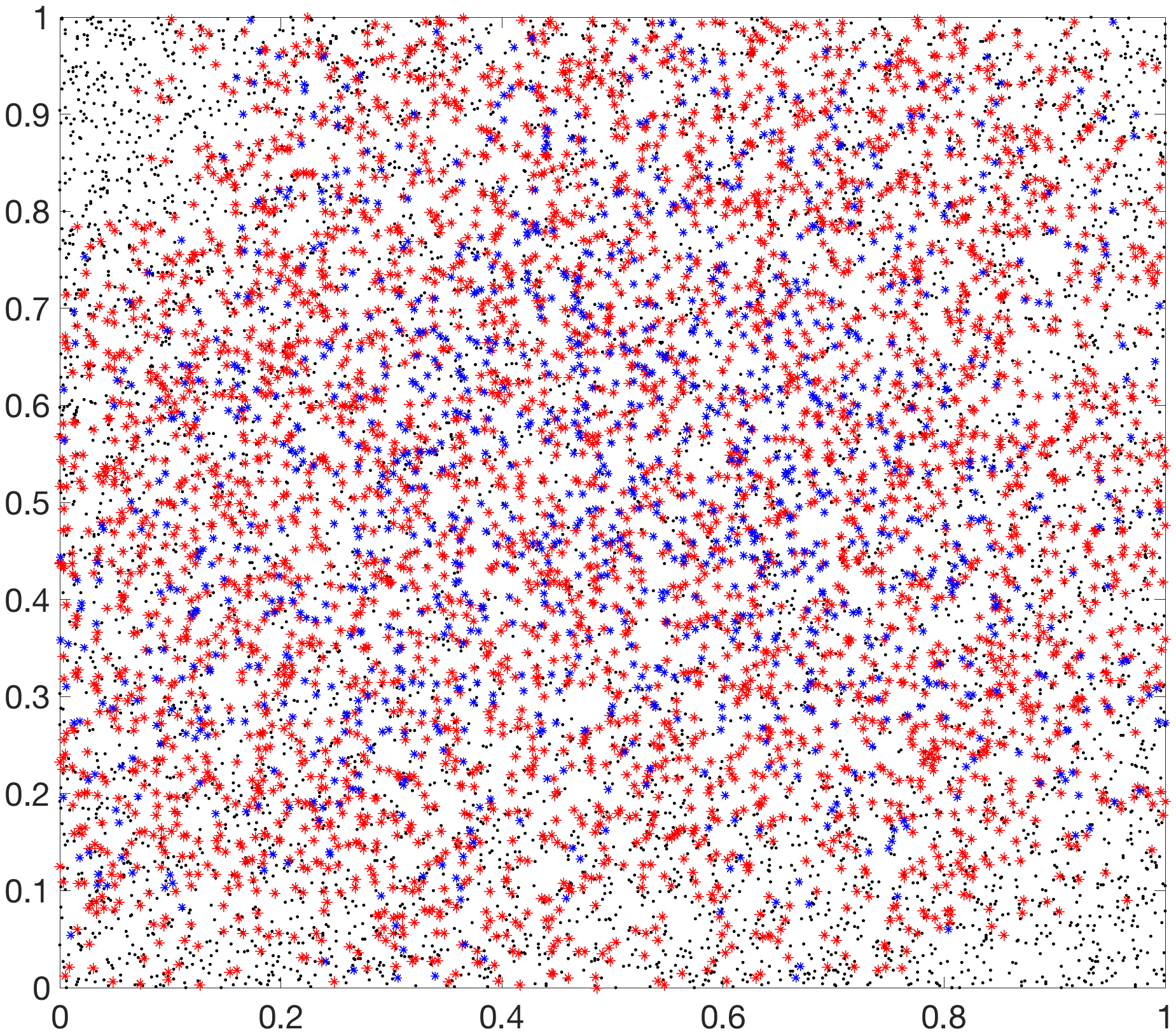}}
   \caption{Simulation of 10000 agents at various time steps with a single agent infected initially, which is located at (0.5,0.5).  
   The states are denoted as \textcolor{black}{$\blacksquare$} Susceptible, \textcolor{red}{$\blacksquare$} Infected, \textcolor{blue}{$\blacksquare$} Recovered.  
   The simulation parameters are defined as: contact tolerance = $\kappa=0.95$, infectivity radius = $\rho_0=0.04$, infection time = $T_I = 30$, recovered time = $T_R = 30$, and spatial step = $\Delta r = 0.001$.
   As time increases, the epidemic spreads as a wave throughout the domain. The inset on (a) shows the radially symmetric neighborhoods of a few infected agents.}
   \label{fig: Simulation}
\end{figure}

%%%%%%%%%%%%%%%%%%%%%%%%%%%%%%%%%%%%%%%%%%%%%%%%%%%%%%%%%%%%%%%%%%%%%%%%%%%%%
\subsection{Globally homogeneous GRR}\label{GlobalSection}
To reduce the number of equations, we can assume a Markovian (time-independent) process.  We further simplify the number of states in our analysis by defining $\I = \bigcup_{i=1}^{T_I} \I_i$ as the infected state, which is  independent of the amount of time spent in the infected state.  
Similarly, we define $\R=\bigcup_{i=1}^{T_R} \R_i$, the total number of recovered agents, regardless of the amount of time spent in the recovered state.
We 
{reduce the number of states in this way when calculating the GRR because we }
are primarily interested in calculating the expected total number of infected and recovered agents at each particular iteration $t$, not the particular stage of the infection or recovery.

Adapting equation (\ref{GeneralGRR}) to our E-AB model, we have the system
\begin{align}
S_{t+1} &= \sum_{\{k:s_t^k = \S\}} W(\S \to \S) + \sum_{\{k:s_t^k = \R\}} W(\R \to \S), \\
I_{t+1} &= \sum_{\{k:s_t^k = \S\}} \prob\left(\x_t^k \in B_t^{\S,\I}\right)W(\S \to \I)
+ \sum_{\{k:s_t^k = \I\}}W(\I \to \I), \label{IGRR}\\
R_{t+1} &= \sum_{\{k:s_t^k = \I\}}W(\I \to \R) + \sum_{\{k:s_t^k  = \R\}} W(\R \to \R).
\label{RGRR}
\end{align}
{Recall that $W(\V\to\U)$ denotes the probability an agent in state $\V$ transitions to state $\U$.}
The total number of agents is constant, so $S_t = N - \left(I_t + R_t\right)$.  This allows the reduction of the above system to just two equations, namely, \eqref{IGRR} and \eqref{RGRR}.  We will now determine the expressions for the probabilities.

We further simplify the derivation by ignoring the effect of the boundary on the infectivity neighborhood, which allows the assumption that the area of the region $\N_t^k$ is independent of $k$ and $t$.  Let $\mu(\N)$ denote the area of any neighborhood $\N_t^k$.  Since our simulation is two-dimensional, we then make the approximation\footnote{
Agent neighborhoods near the boundary will have smaller area, since, by definition, the neighborhood is contained in $\Omega$.  However, we assume that since the initially infected agent is located in the center of the region, there are sufficiently many infected agents away from the boundary to make this simplification reasonable.} 
that $\mu(\N):=\mu(\N_t^k) = \pi \rho_0^2$ for all $k$ and $t$.  It follows that the probability that the $k$th agent is located in the neighborhood of the $j$th agent is
\begin{equation}
\prob(\x_t^k \in \N_t^j) = \frac{\mu(\N)}{\mu(\Omega)}, \quad \forall \j,
\end{equation}
the ratio of the area of the infectivity neighborhood and the area of the region.

For any susceptible agent $k$ to transition to state $\I$, it is sufficient that $\x_t^k \in B_t^{\S,\I}$.  
If we assume that the transition probability $W(\S \to \I)$ does not depend on the number of infectivity neighborhoods and that the infectivity neighborhoods are uniformly distributed within $\Omega$, then by the multiplication rule of independent events,
\begin{equation}
\prob(\x_t^k \notin B_t^{\S,\I}) = \left( 1 - \frac{\mu(\N)}{\mu(\Omega)}\right)^{I_t}.
\end{equation}
It follows that the probability of $\x_t^k$ being located in an $\S \to \I$ transition neighborhood is then
\begin{equation}
\prob(\x_t^k \in B_t^{\S,\I}) = 1 - \left( 1 - \frac{\mu(\N)}{\mu(\Omega)}\right)^{I_t}.
\label{GlobalS}
\end{equation}
Since $W(\S \to \I)$ depends on agent $k$ being located in at least one infectivity neighborhood, and does not change if agent $k$ is located in more than one infectivity neighborhood, it follows that $W(\S \to \I) = 1-\kappa$, where $\kappa\in[0,1]$ is the contact rate.

Then, for any $k$ such that $s_t^k = \S$, by inserting \eqref{GlobalS} into \eqref{IGRR} we have:
\begin{equation}
\begin{split}
\prob(s_{t+1}^k = \I) &= \prob(\x_t^k \in B_t^{\S,\I})W(\S \to \I) \\
&= 
\underbrace{ \left\{ 1 - \left( 1 - \frac{\mu(\N)}{\mu(\Omega)}\right)^{I_t} \right\}}_{\substack{\text{probability in at least}\\ \text{one infectivity region}}}
\underbrace{ (1-\kappa)}_{\substack{\text{probability}\\ \text{becoming}\\ \text{infected}}}.
\label{GlobalSEq}
\end{split}
\end{equation}
Moreover, by assuming the cumulative time spent in infected states is uniformly distributed, we have for any agent $k$ such that $s_t^k = \I$,
\begin{align}
W(\I \to \R) &= 1/T_I, \label{GlobalR}\\
W(\I \to \I)&=1-1/T_I,
\label{GlobalI}
\end{align}
where $T_I$ is the time spent in the infective state.
This assumption is valid for a large number of agents and for a sufficiently large number of iterations.
Inserting equations (\ref{GlobalSEq}), (\ref{GlobalR}), and (\ref{GlobalI}) into (\ref{IGRR}), we have
\small{
\begin{equation}
\underbrace{I_{t+1}}_{{\substack{\text{total}\\ \text{infected}\\ \text{agents} \\ \text{at time }t+1}}} = 
\underbrace{(N-I_t-R_t)}_{\substack{\text{total}\\ \text{susceptible}\\ \text{agents} \\ \text{at time }t}}
\underbrace{\left\{ 1 - \left( 1 - \frac{\mu(\N)}{\mu(\Omega)}\right)^{I_t} \right\}}_{ \substack{\text{probability in at least}\\ \text{one infectivity region}}}
\underbrace{(1-\kappa)}_{\substack{\text{probability}\\ \text{becoming}\\ \text{infected}}}
+ 
\underbrace{\left\{1-\frac{1}{T_I}\right\}}_{\substack{
\text{probability remain} \\
\text{in infected state}}}
\underbrace{I_t}_{\substack{\text{total}\\ \text{infected}\\ \text{agents} \\ \text{at time }t}}. 
\end{equation}
}
That is, the expected number of infected agents at $t+1$ is the sum of two terms.  The first term is the product of the expected number of susceptible agents at time $t$ multiplied by the probability a susceptible agent transitions to state $\I$.  The second term is the expected number of infected agents at $t$ times the probability that an infected agent remains in state $\I$.
Similarly, by assuming the time in recovered states is uniformly distributed, we have for any agent $k$ such that $s_t^k = \R$, the probability of staying in state $\R$ is
\begin{equation}
W(\R \to \R) = 1 - 1/T_R.
\label{GlobalRR}
\end{equation}  
Then, inserting (\ref{GlobalR}) and (\ref{GlobalRR}) into (\ref{RGRR}) we have
\begin{equation}
R_{t+1} = \frac{1}{T_I}I_t + \left(1 - \frac{1}{T_R}\right) R_t,
\end{equation}
and the expected number of recovered agents at iteration $t+1$ is the sum of two terms.  The first term is the expected number of infected agents at time $t$ times the probability an infected agent transitions to state $\R$.  The second term is the expected number of recovered agents at $t$ times the probability a recovered agent remains in state $\R$.

Since $T_I$ and $T_R$ will be explicitly defined, we can easily find a $q\in \mathbb{R}$ such that $T_R=qT_I$.  We then have the following E-AB GRR:
\begin{align}
I_{t+1} &= (N-I_t-R_t)\left\{ 1 - \left( 1 - \frac{\mu(\N)}{\mu(\Omega)}\right)^{I_t} \right\}(1-\kappa)
+ \left(1-\frac{1}{T_I}\right)I_t :=H(I_t,R_t), 
\label{Global GRRI} \\
R_{t+1} &= \frac{1}{T_I}I_t + \left(1 - \frac{1}{qT_I}\right) R_t:=G(I_t,R_t).
\label{Global GRRR}
\end{align}
Since $S_t = N - \left(I_t + R_t\right)$, we have recurrence formulas for the expected agent densities in each state at each iteration.
With our GRR, we now have a general framework to further analyze the behavior of the system. Note that we identify \eqref{Global GRRI}-\eqref{Global GRRR} as globally homogeneous since we have assumed the infectivity neighborhoods are uniformly distributed in the domain with the same constant area.

%%%%%%%%%%%%%%%%%%%%%%%%%%%%%%%%%%%%%%%%%%%%%%%%%%%%%%%%%%%%%%%%%%%%%%%%%%%%%
\subsection{Fixed point analysis for globally homogeneous GRR}\label{FixedPointSIRSection}
We can now calculate the stability of the fixed points of the globally homogeneous GRR by finding all solutions to the system that simultaneously solve $I_{t+1}-I_t=0$ and $R_{t+1} - R_t=0$.
That is, we need to find all solutions to the system
\begin{equation}
\small
\begin{bmatrix}
(N-I_t-R_t)\left\{ 1-\left(1-\mu(\N)\right)^{I_t}\right\}(1-\kappa)-\frac{1}{T_I}I_t \\
\frac{1}{T_I}I_t - \frac{1}{qT_I}R_t
\end{bmatrix}
=
\begin{bmatrix} 0 \\ 0 \end{bmatrix}.
\label{SIRFixedSetup}
\end{equation}
We have two fixed points.  One is the trivial fixed point, $(I,R)=(0,0)$.  The other is the point along the line $R=qI$ that solves the fixed point problem
\begin{equation}
(N-(1+q)I)\left\{ 1-\left(1-\mu(\N)\right)^{I}\right\}(1-\kappa)-\frac{1}{T_I}I = I.
\end{equation}
This second fixed point has to be computed numerically.

We will analyze the fixed point $(0,0)$ using two-dimensional perturbation theory where details can be found in 
\cite{Deutsch,LinSegel}.
The Jacobian matrix of the E-AB GRR is
\begin{align*}
J= {\tiny \begin{bmatrix}
-(1-\kappa)\Big\{(N-I-R)\left( (1-\mu(\N))^I \ln(1-\mu(\N))\right) + (1-(1-\mu(\N))^I)\Big\}-\frac{1}{T_I} & -(1-\kappa) (1-(1-\mu(\N))^I) \\
\frac{1}{T_I} & 1-\frac{1}{qT_I}
\end{bmatrix} }.
\end{align*}
Evaluating at $(0,0)$ gives us 
\begin{equation}
J\Big|_{(I,R)=(0,0)} = \begin{bmatrix}
-N(1-\kappa)\ln\left(1-\mu(\N)\right)- \frac{1}{T_I} & 0 \\
\frac{1}{T_I} & 1-\frac{1}{qT_I}
\end{bmatrix}.
\end{equation}
The eigenvalues are $\lambda_1 = 1-\frac{1}{qT_I}$ and $\lambda_2=-N(1-\kappa)\ln(1-\mu(\N))-\frac{1}{T_I}$.  
Clearly, since $qT_I = T_R>0$ we know $|\lambda_1|<1$.  
Now, suppose $\lambda_2<1$.  It follows that $\alpha > 1 - \exp\left(\frac{1+1/T_I}{N(1-\kappa)}\right)$.  
That is, $\frac{\mu(\N)}{\mu(\Omega)} > 1 - \exp\left(\frac{1+1/T_I}{N(1-\kappa)}\right)$.  
We know $\mu(\N)\ll \mu(\Omega)$ and it is reasonable to assume that $N$ is sufficiently large such that $N(1-\kappa)>2$.  This contradicts the inequality. It must follow that $\lambda_2>1$.  Thus, we have that $(0,0)$ is a saddle point that is only stable along the nullcline $I=0$.

Since we do not have an explicit solution of the second fixed point, we cannot perform the same computation as we did for $(0,0)$.  However, we know that the following are bounded: $H$ and all derivatives of $H$, 
{the expected number of infected agents at the next time step from equation \eqref{Global GRRI}}, and the domain. Additionally, since $I$ is repelled by $(0,0)$, we can infer that the second fixed point is stable.

Similar to a differential equation SIR model, we are able to obtain fixed points. However, the fixed points are different and there are not direct comparisons since we assume a moving population with dynamic contacts that allow for infection whereas a differential equation assumes a well-mixed population.

%%%%%%%%%%%%%%%%%%%%%%%%%%%%%%%%%%%%%%%%%%%%%%%%%%%%%%%%%%%%%%%%%%%%%%%%%%%%%
\subsection{Locally homogeneous GRR}\label{LocalSection}
When deriving the globally homogeneous E-AB GRR, (\ref{Global GRRI}) and (\ref{Global GRRR}), we assumed that the infectivity neighborhoods were uniformly distributed throughout $\Omega$.  For this test case, we initialize one infected agent, $s_0^1 = \I$, such that it is initially located in the center of the region $\x_0^1 = (0.5,0.5)$.  However, from observation of simulations, such as in Figure \ref{fig: Simulation} (or intuition), we know there is a wave of infectivity propagating from this initial infected agent.  The susceptible agents that agent 1 infects must be located in its neighborhood $\N^{1}$.  Future infected agents will be located in those neighborhoods.  So rather than generalize a uniform distribution of infected agents, we should modify the GRR to account for the infection wave front.

We then need to create a sequence of regions 
$\left\{\tB_0^{\S,\I}, \tB_1^{\S,\I}, ...\right\}$
, where $\tB_0^{\S,\I} = \N^1$ and $\tB_t^{\S,\I}$ is the smallest connected region containing the infection front at time $t$.  
For notation, we will use tildes above variables to denote variables and functions specifically defined for the locally homogeneous case.
\begin{definition} 
$\displaystyle \tB_t^{\S,\I} = \inf_A \left\{A\subset \Omega: A \text{ is connected and } \cup_{\{k:s_t^k=\I\}} \N_t^k \subset A\right\}$.
\end{definition}

Suppose agent $k$ is such that $s_t^k = \S$ at iteration $t$.  We have the following conditional probability that an agent is located in the $\S\to \I$ transition neighborhood, $B_t^{\S,\I}$, given that the agent is within the infection front, $\tB_t^{\S,\I}$:
\begin{equation}
\prob\left(\x_t^k \in 
B_t^{\S,\I} \Big| \x_t^k \in \tB_t^{\S,\I}\right) 
= 1-\left(1 - \frac{\mu(\N)}{\mu\left(\tB_t^{\S,\I}\right)}\right)^{I_t}.
\label{modify1}
\end{equation}
In general, for regions $\N$ and $\tB_t^{\S,\I}$, the probability is given as
$\prob\left(\x_t^k \in \tB_t^{\S,\I}\right) = \frac{\mu\left(\tB_t^{\S,\I}\right)}{\mu(\Omega)}.$

Using Bayes' theorem \cite{Talbott}, we have that the locally homogeneous probability of an agent being in the $\S\to \I$ transition neighborhood is
\begin{equation}
\begin{split}
\prob\left(\x_t^k \in B_t^{\S,\I}\right)
&= \prob\left(\x_t^k \in B_t^{\S,\I} \Big| \x_t^k \in \tB_t^{\S,\I}\right) \prob\left(\x_t^k \in \tB_t^{\S,\I}\right), \\
&= \left\{1-\left(1 - \frac{\mu(\N)}{\mu\left(\tB_t^{\S,\I}\right)}\right)^{I_t}\right\}\frac{\mu\left(\tB_t^{\S,\I}\right)}{\mu(\Omega)}.
\end{split}
\label{modify3}
\end{equation}
Inserting \eqref{modify3} into \eqref{GlobalSEq}, our locally homogeneous E-AB GRR is
\begin{equation}
\begin{split}
\tilde I_{t+1} &= \left(N-\tilde I_t - \tilde R_t\right)
\left\{1-\left(1 - \frac{\mu(\N)}{\mu\left(\tB_t^{\S,\I}\right)}\right)^{I_t}\right\}\frac{\mu\left(\tB_t^{\S,\I}\right)}{\mu(\Omega)}(1-\kappa)  \\
&\qquad \qquad +\left(1-\frac{1}{T_I}\right)\tilde I_t =:\tilde H(\tilde I_t, \tilde R_t),\\
\tilde R_{t+1} &= \frac{1}{T_I}\tilde I_t + \left(1 - \frac{1}{qT_I}\right) \tilde R_t =: \tilde G(\tilde I_t, \tilde R_t).
\end{split}
\label{LocalGRRR}
\end{equation}
Recall that the tilde denotes values associated with the locally homogeneous GRR.

We derived this GRR by focusing on early dynamics.  But how does the non-uniform assumption of the infection front affect the stability using the locally homogeneous GRR compared with the globally homogeneous GRR?
\begin{theorem}
If $\tB_t^{\S,\I} \to \Omega$ as $t\to +\infty$ and $\alpha$ is a fixed point of $H$, then $\alpha$ is a fixed point of $\tilde H$.  Moreover, $\alpha$ has the same stability conditions for $H$ and $\tilde H$.
\label{LocalSIRTHM}
\end{theorem}
\begin{proof}
Suppose that $\lim_{t\to +\infty}I_t = \alpha$ and suppose that $\lim_{t\to +\infty}\tilde I_t$ exists.  
Then, since $\mu\left(\tB_t^{\S,\I}\right) \to \mu(\Omega)$ as $t \to + \infty$, we have that 
$\lim_{t\to +\infty}\left(1-\frac{\mu(\N)}{\mu\left(\tB_t^{\S,\I}\right)}\right)^{\tilde I_t} =
\lim_{t\to +\infty}\left(1-\frac{\mu(\N)}{\mu(\Omega)}\right)^{\tilde I_t}$.
Plugging into \eqref{Global GRRI} and taking the limit,
\begin{small}
\begin{align*}
\lim_{t\to +\infty}\tilde H(\tilde I_t)
&=\lim_{t\to +\infty}\left[ (N-\tilde I_t)\left(1 - \frac{\mu(\N)}{\mu(\Omega)}\right)^{\tilde I_t}(1-\kappa)+\left(1 - \frac{1}{qT_I}\right)\tilde I_t \right]\\
&=\lim_{t\to +\infty} H(\tilde I_t) = \alpha.
\end{align*}
\end{small}
Moreover, $\mu(\tB_{\S,\I}^t)\to \mu(\Omega)$ for fixed $\alpha$ and
it is clear that
$\frac{\partial \tilde H}{\partial \tilde I}\Big|_\alpha \to \frac{\partial H}{\partial I}\Big|_\alpha$, 
$\frac{\partial \tilde H}{\partial \tilde R}\Big|_\alpha \to \frac{\partial H}{\partial R}\Big|_\alpha$,
$\frac{\partial \tilde G}{\partial \tilde I}\Big|_\alpha \to \frac{\partial G}{\partial I}\Big|_\alpha$, and 
$\frac{\partial \tilde G}{\partial \tilde R}\Big|_\alpha \to \frac{\partial G}{\partial R}\Big|_\alpha$.
Since the stability condition depends on the Jacobian, and the Jacobian of the locally homogeneous region approaches the Jacobian of the globally homogeneous region as $t\to +\infty$, the long term stability conditions must be the same.
\end{proof}

From Theorem \ref{LocalSIRTHM} we know that $(\tilde H, \tilde G)$ has the same fixed points as found in Section \ref{FixedPointSIRSection} with the same stability conditions.  We thus reduced the problem to capturing an explicit formula for $\mu\left( \tB_t^{\S,\I} \right)$. 
Before, we made the simplifying assumption that the infected agents were distributed uniformly throughout the region, so we did not need to incorporate any spatial characteristics into the globally homogeneous GRR.  Now, we need to capture the infection front dynamics in order to explicitly calculate the area of the infectivity region, $\mu\left( \tB_t^{\S,\I} \right)$, in the locally homogeneous case.

\begin{figure}[htb!]
    \centering
    \includegraphics[width=5cm]{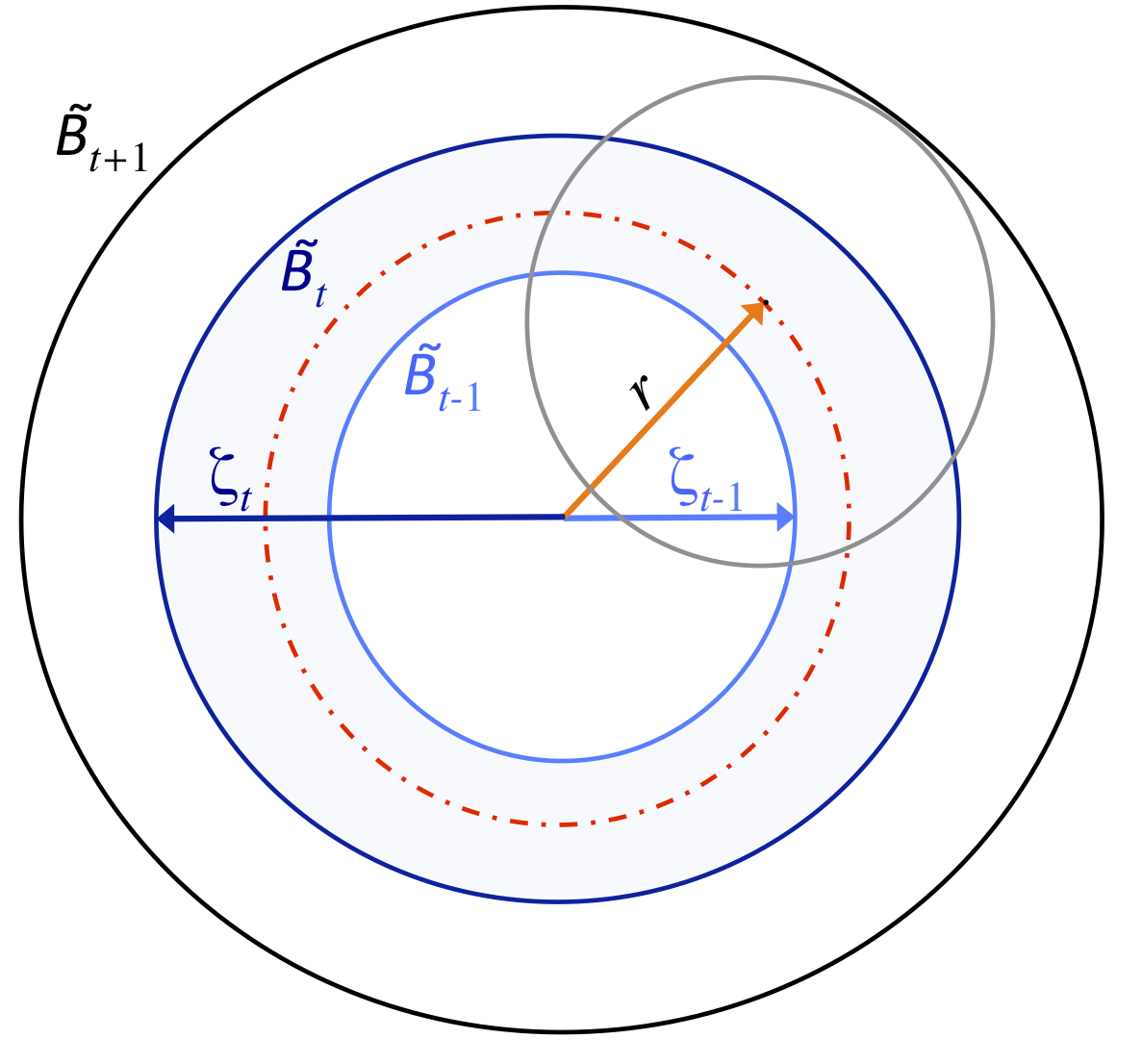} 
    \caption{Diagram demonstrating how the locally homogeneous $\tB_{t+1}^{\S,\I}$ region depends on $\tB_t^{\S,\I}$ and $\tB_{t-1}^{\S,\I}$ regions.  The initially infected agent at $t=0$ is located in the center of the $\tB_t^{\S,\I}$ regions, which expand outward with radii $\zeta_t$.  We assume newly infected agents lie on the radial center of mass of the region $\tB_t^{\S,\I}\setminus \tB_{t-1}^{\S,\I}$, denoted with the dashed line, which is a distance $r$ from the initially infected agent.}
    \label{B_kImage}%
\end{figure}
For our formula, we will assume that newly infected agents are expected to be in the region $\tB_t^{\S,\I}\setminus \tB_{t-1}^{\S,\I}$ and are moving a fixed distance $\Delta r$. Further, we assume the expected location of the newly infected agents will lie on the circle that is the radial center of mass of $\tB_t^{\S,\I}\setminus \tB_{t-1}^{\S,\I}$, a distance $r$ from the initially infected agent location as shown in Figure \ref{B_kImage}.  The radius of the infectivity front region $\tB_t^{\S,\I}$ at time $t$ is denoted as $\zeta_t$.  We then have the following expected radius of $\tB_{t+1}^{\S,\I}$: 
\begin{equation}
\zeta_{t+1} = \rho_0 + \sqrt{
\frac{(\zeta_t + \delta_{out}(\Delta r))^2 + (\zeta_{t-1}-\delta_{in}(\Delta r))^2}{2}},
\label{Bkeq}
\end{equation}
where $\delta_{in}$ and $\delta_{out}$ are functions of the expected distance an infected agent travels towards the center of $\tB_t^{\S,\I}$ and out of the region $\tB_t^{\S,\I}$, respectively. 
For our simulations and derivation, we assume 
$\delta_{out}=\Delta r$ and $\delta_{in}=\Delta r$.  
Even though the simulation is a Markov process, our analytical solution, which calculates the area $\mu\left( \tilde{B}_{t+1}^{\S,\I}\right)$ using $\zeta_{t+1}$, is not, since it relies on information at iterations $t$ and $t-1$.  Clearly, the simulation is a Markovian process but the GRR does not have to be for this analysis. In fact, it  can belong to a larger class of processes than the underlying AB model.

By the following theorem, we know that equation (\ref{Bkeq}) satisfies the premise of Theorem \ref{LocalSIRTHM}.   
\begin{theorem}
If $\rho_0>0$ and for all $j$ such that $\mu\left( \tB_j^{\S,\I}\right) \neq \mu(\Omega)$ there exists an agent $k$ with $s_j^k = \S$ such that $\x_j^k \notin \tB_j^{\S,\I}$, then $\exists \hat t\in \mathbb{N}$ such that $\mu\left(\tB_t^{\S,\I}\right) = \mu(\Omega)$ for all $t>\hat{t}$ with radius $\zeta_{t}$ as defined in (\ref{Bkeq}). 
\label{Bkthm}
\end{theorem}

The above theorem essentially states that if the infection does not ``die out,'' then the $\S \to \I$ transition neighborhood, $\tilde{B}_j^{\S,\I}$, eventually covers the entire region of interest $\Omega$.  
The proof is clear, since $\Omega$ is bounded.
As we will see in Section \ref{Sec: Numerical}, we are able to approximate early behavior more accurately with the locally homogeneous GRR, while still being able to evaluate and determine the stability of fixed points with the simpler equations of the globally homogeneous GRR.

\subsection{Extensions to more complex neighborhoods}\label{inclusion}
\begin{figure}%[H]
\begin{center}
     \includegraphics[width=0.35\textwidth]{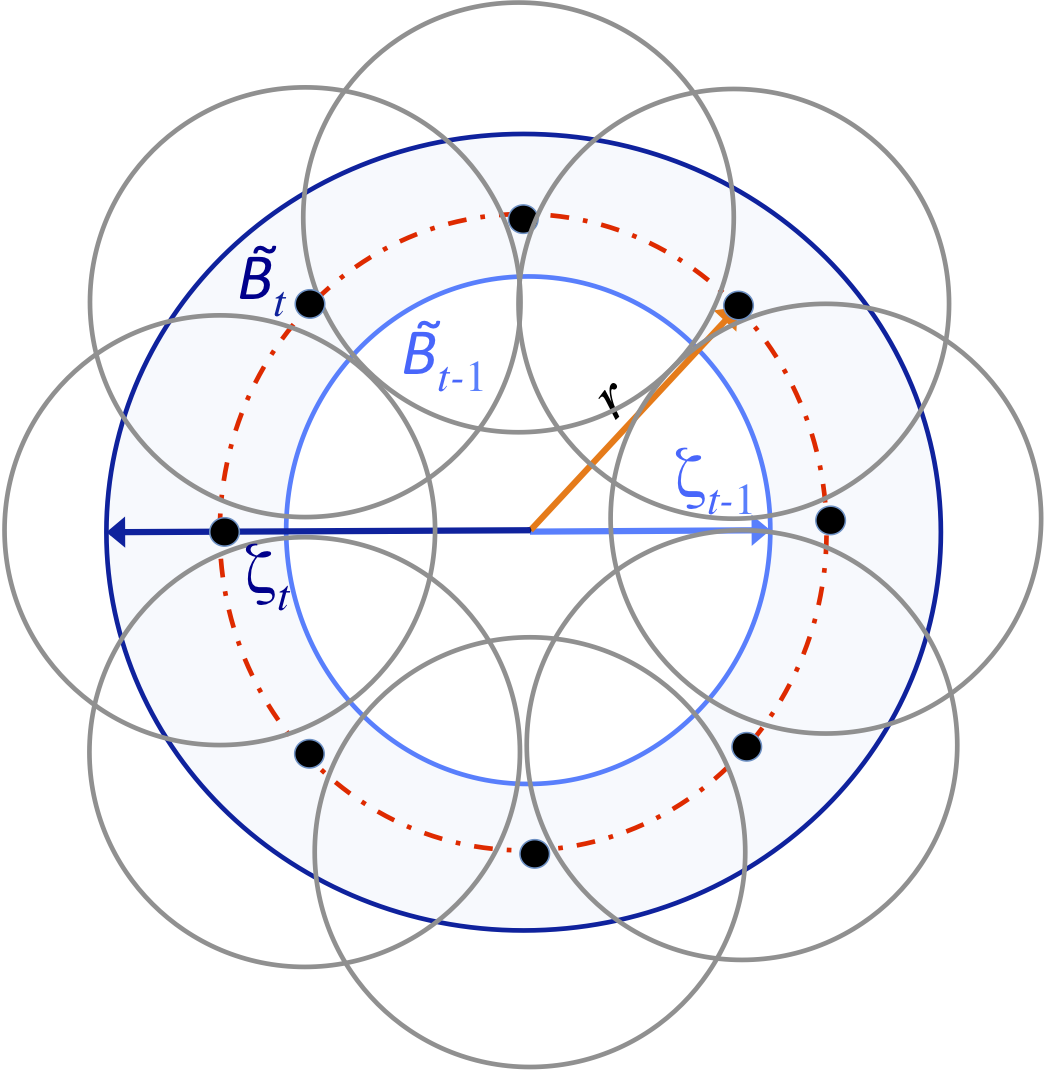}
   \end{center}
  \caption{Infectivity neighborhood depends on the number $n$ of newly infected agents. The radial center of mass of the region $\tB_t^{\S,\I} \setminus \tB_{t-1}^{\S,\I}$ (dashed line) is a distance $r$ from the initially infected agent. Newly infected agents are  distributed uniformly along the radial center of mass and the new infection front radius $\zeta_{t+1}$ will depend on the total area of the infectivity neighborhoods outside $\tB_t^{\S,\I}$.}
  \label{InfectionOverviewInPaper}
\end{figure}
As long as the infection front in the E-AB approaches $\partial \Omega$ and $\tB_t^{\S,\I} \to \partial \Omega$, Theorem \ref{LocalSIRTHM} holds.  One could derive a formula for $\tB_t^{\S,\I}$ that more closely approximates the initial phases of the infection spread.  
Rather than assume that the new infection front extends approximately $\Delta r$ from the mean center of mass as in Figure \ref{B_kImage}, we can assume that the infectivity radius expansion depends on the number of newly infected agents in $\tB_t^{\S,\I} \setminus \tB_{t-1}^{\S,\I}$, as shown in Figure \ref{InfectionOverviewInPaper}.

Further details regarding the calculation and derivation for this case of $\tB_t^{\S,\I}$ can be found in  Appendix \ref{Section:Sparse}.  In this example, we illustrate that, by relaxing assumptions, one can derive other expressions calculating the area of the infectivity front radius $\zeta_t$ that may decrease the error of the locally homogeneous GRR during the early stages of the epidemic.  Moreover, we know that as long as the new formulation of $\zeta_t$ maintains the suppositions of Theorem \ref{LocalSIRTHM}, the long term dynamics will be captured.

%%%%%%%%%%%%%%%%%%%%%%%%%%%%%%%%%%%%%%%%%%%%%%%%%%%%%%%%%%%%%%
%%%%%%%%%%%%%%%%%%%%%%%%%%%%%%%%%%%%%%%%%%%%%%%%%%%%%%%%%%%%%%
%%%%%%%%%%%%%%%%%%%%%%%%%%%%%%%%%%%%%%%%%%%%%%%%%%%%%%%%%%%%%%
%%%%%%%%%%%%%%%%%%%%%%%%%%%%%%%%%%%%%%%%%%%%%%%%%%%%%%%%%%%%%%
\section{Numerical results}\label{eca}
\label{Sec: Numerical}
Results for the E-AB model with $N=10000$ initialized agents are shown in Figure 
\ref{SIRCA_Simulations}.  
Note that each simulation curve on the plot is the average of 1000 E-AB simulations whereas the curves based on the Globally 
\begin{figure}
    \centering
    \subfloat[$N=10000, \kappa=0.6, T_I=T_R=30$]{{
    \includegraphics[width=0.325\textwidth]{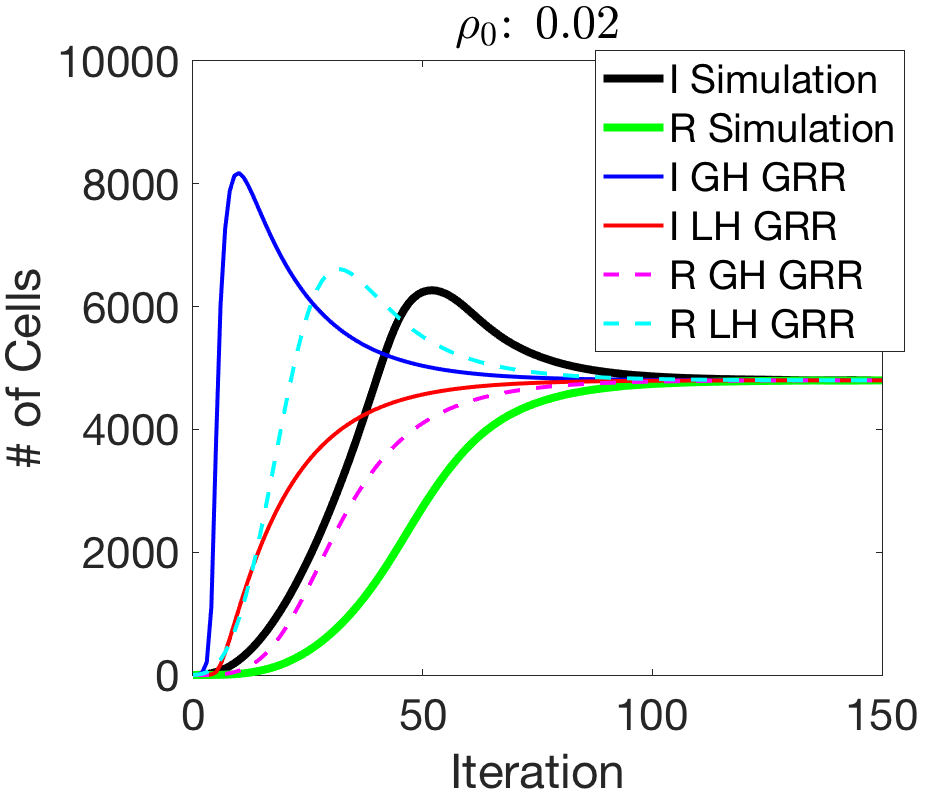}
    \includegraphics[width=0.3\textwidth]{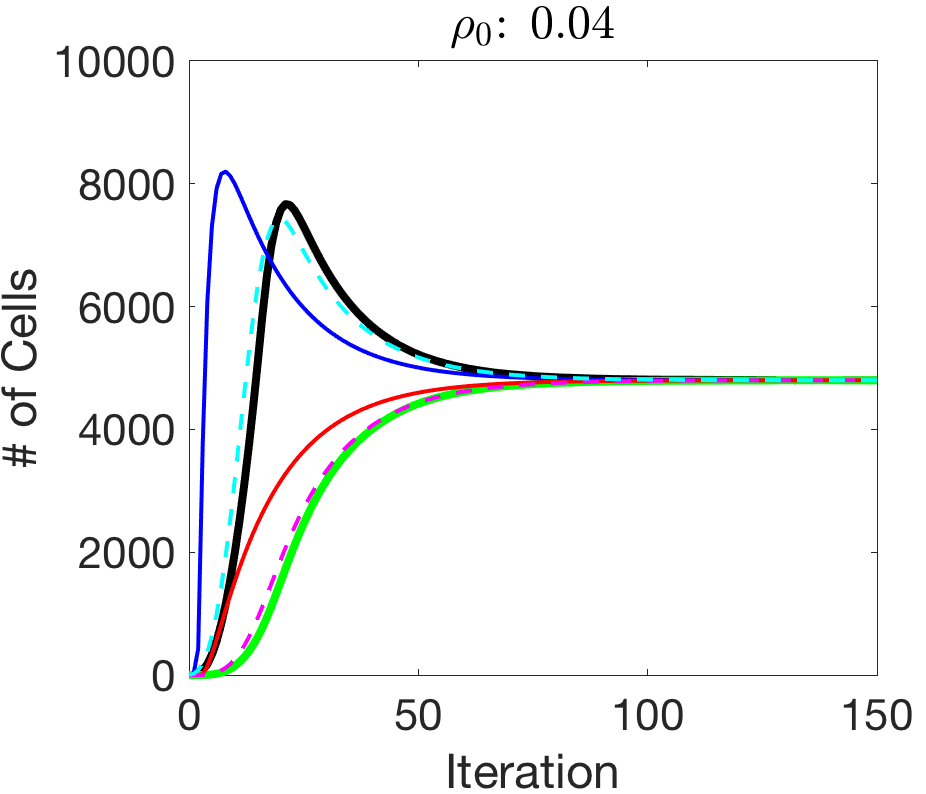}
    \includegraphics[width=0.3\textwidth]{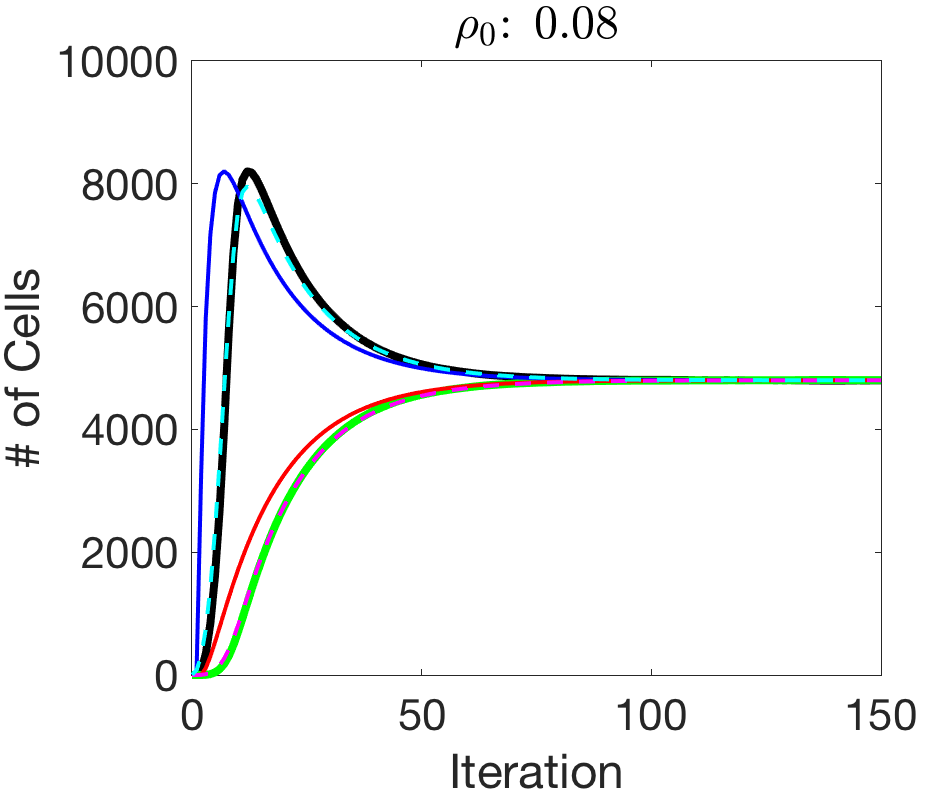}
    }} \\
    \subfloat[$N=10000, \kappa=0.8, T_I=T_R=30$]{{
            \includegraphics[width=0.325\textwidth]{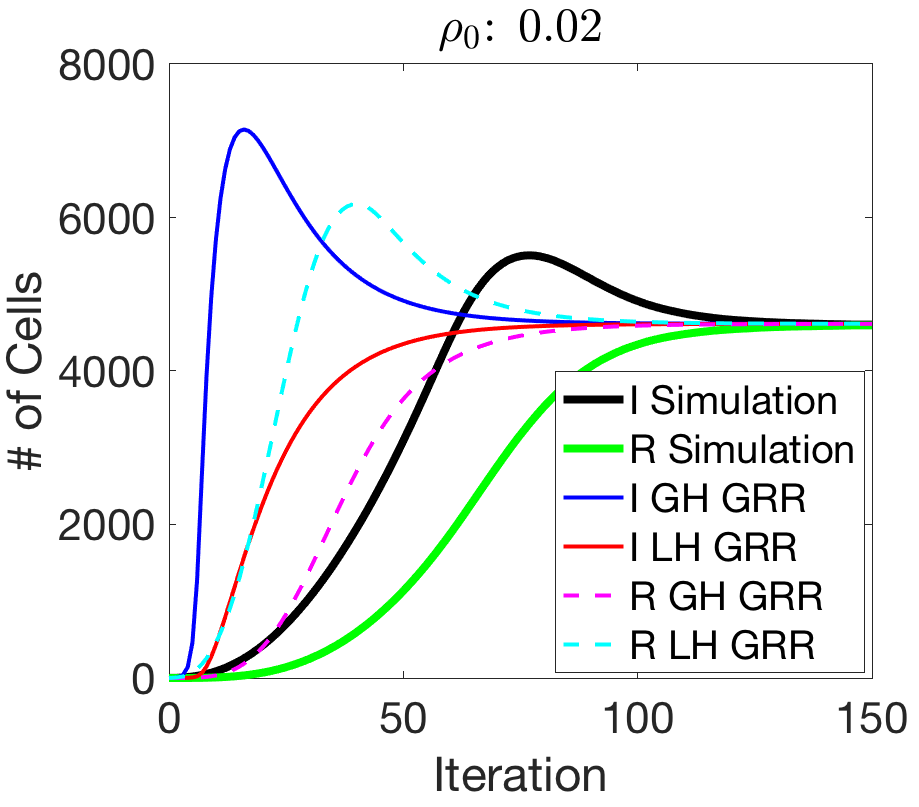}
    \includegraphics[width=0.3\textwidth]{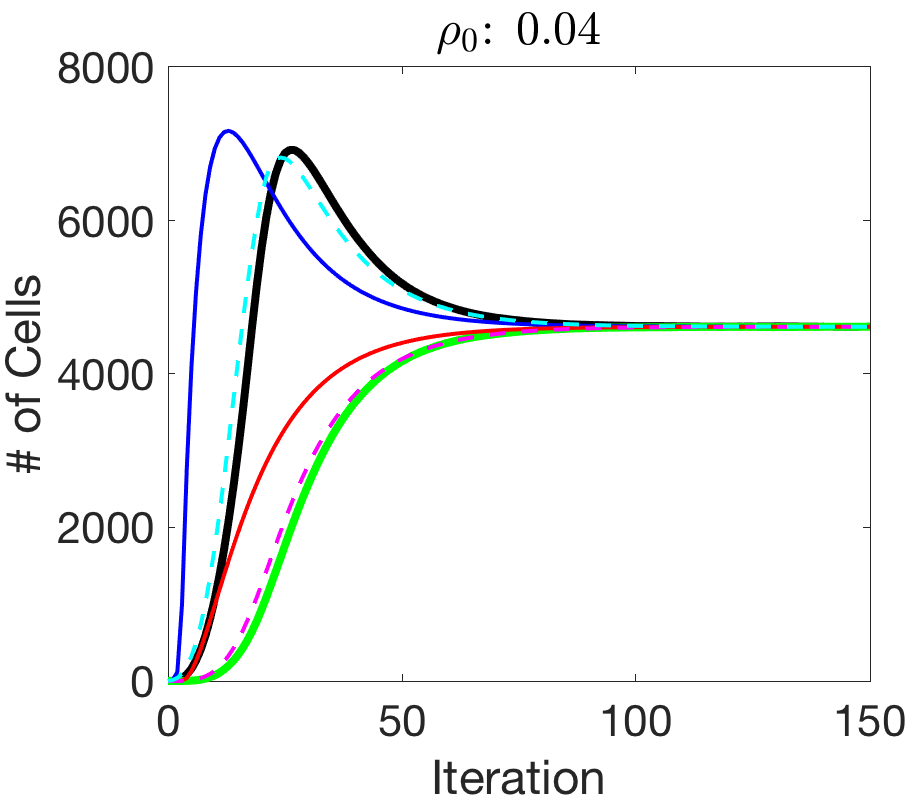}
    \includegraphics[width=0.3\textwidth]{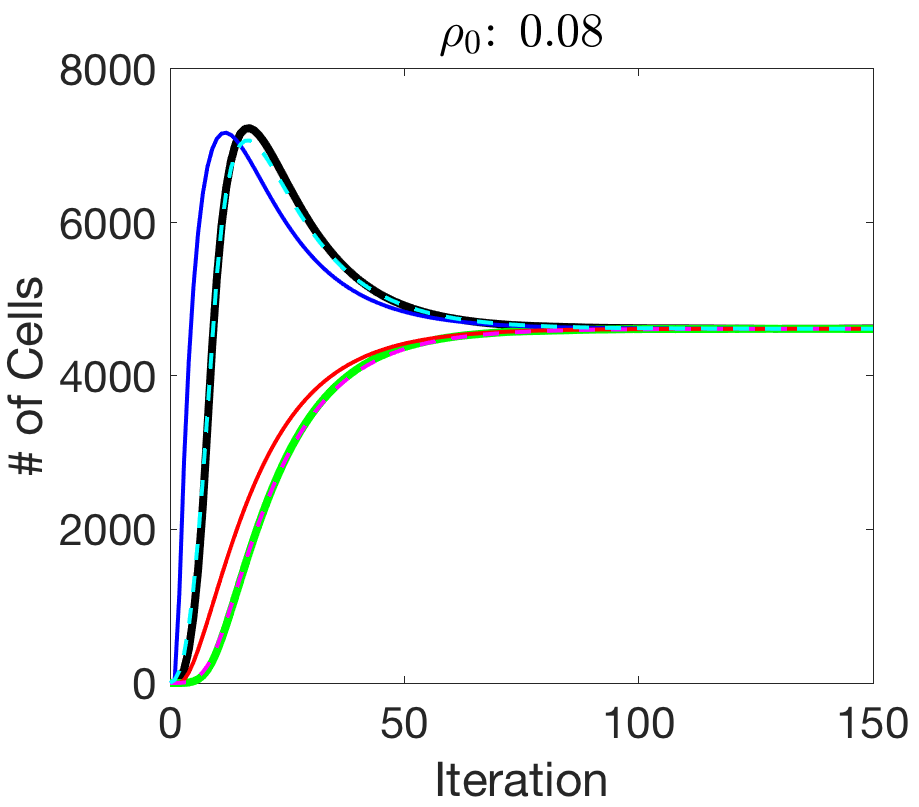}
    }} \\
    \subfloat[$N=10000, \kappa=0.8, T_I=30, T_R=45$]{{
            \includegraphics[width=0.325\textwidth]{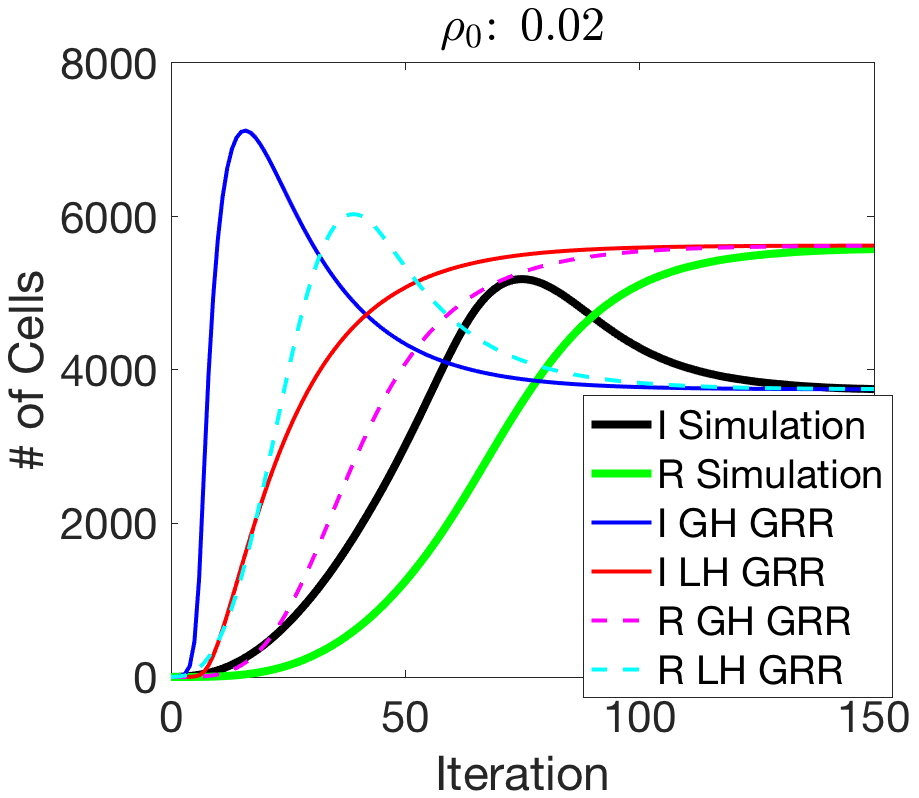}
    \includegraphics[width=0.3\textwidth]{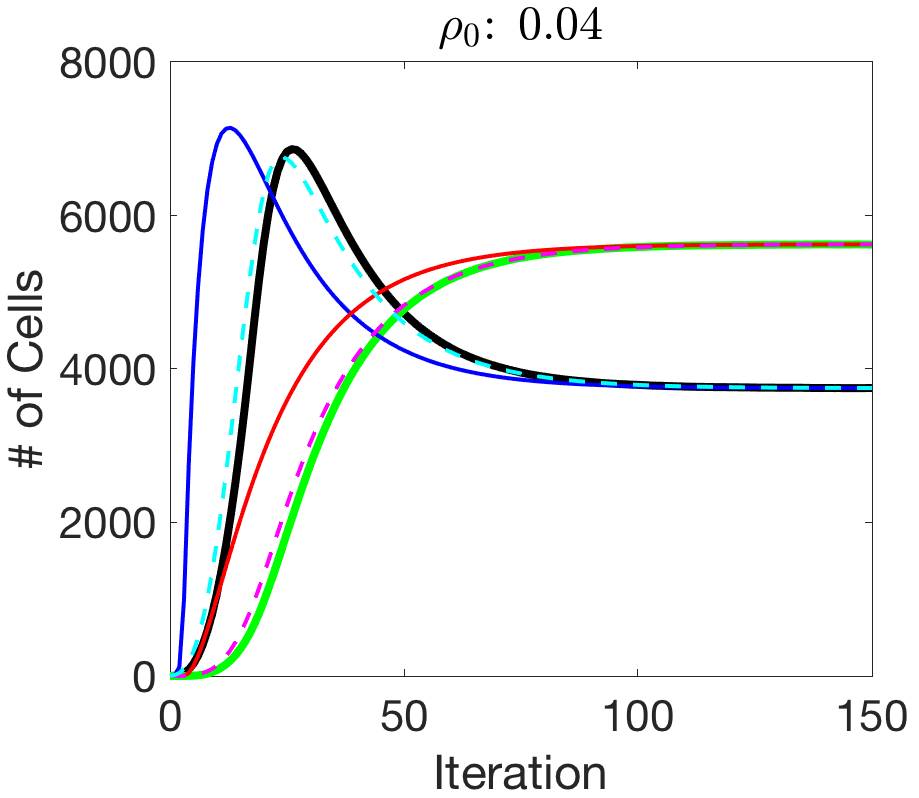}
    \includegraphics[width=0.3\textwidth]{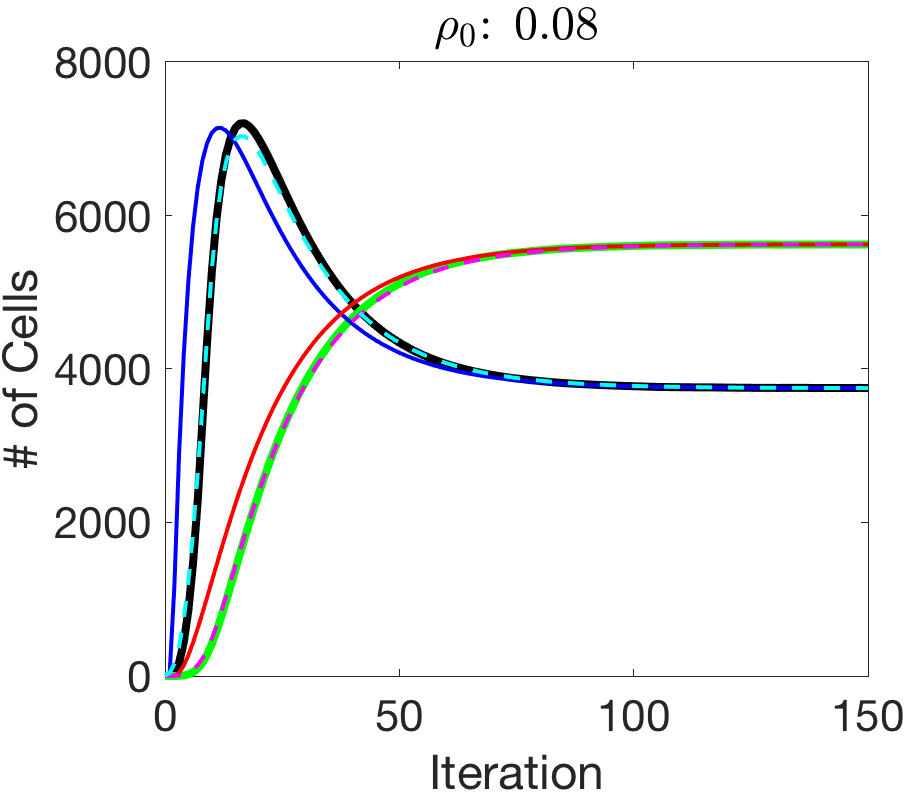}
    }}
    \caption{Comparing the average value of 1000 E-AB simulations with results from the globally and locally homogeneous  GRRs calculated from \eqref{Global GRRR} and \eqref{LocalGRRR}, respectively. The time to remain infected is set to $T_I=30$, while the recovery time is $T_R=30$ in (a)-(b) and $T_R=45$ in (c). Each plot corresponds to a different  infectivity radii parameter $\rho_0$ with contact tolerance $\kappa=0.6$ in (a) and $\kappa=0.8$ in (b)-(c). Note that for the globally homogeneous case (labeled as GH GRR), the infectivity neighborhood has fixed radius of $\rho_0$, whereas the locally homogeneous case (labeled as LH GRR) has a variable radii $\zeta_t$ at each iteration $t$ as given in \eqref{Bkeq}, which is a function of $\rho_0$.}
    \label{SIRCA_Simulations}
\end{figure}
and locally homogeneous GRRs are from solving \eqref{Global GRRR} and \eqref{LocalGRRR}, respectively.  
In the Figure, we observe agreement of long term behavior of the simulations with both the globally and locally homogeneous GRRs. For example, in Figure \ref{SIRCA_Simulations}, the left hand column of each row corresponds to the case where $\rho_0 = 0.02$. In the left column of (c), the average of the E-AB simulations for the fixed points or long term behavior is 3877.1 infected agents and 5790.9 recovered agents. Upon calculation, the relative error between the simulated fixed points and the GRR fixed points is $\mathcal{O}\left(10^{-4}\right)$.  Further, the early time dynamics of the infected and recovered populations with the GRR estimates have behavior similar to that of the E-AB simulations. 

However, the early infection dynamics of the GRRs do not exactly match the simulations for all cases. In Figure \ref{SIRCA_Simulations}(a)-(b), for a contact tolerance $\kappa=0.6$ and $\kappa=0.8$ (characterizing how easily an agent becomes infected), we observe that as the infectivity radius increases, the GRRs are able to more accurately capture the early time dynamics of the E-AB simulations. For an infectivity neighborhood of radius $\rho_0=0.02$, it is likely that there is not a sufficient number of agents in the region to accurately capture the early time dynamics of infectivity.  We do observe that the locally homogeneous GRR provides a better approximation to the E-AB simulations in comparison to the globally homogeneous GRR. Similar trends are observed in Figure \ref{SIRCA_Simulations}(c), where the recovery time $T_R$ is increased. 

To explicitly define how much ``better'' the locally homogeneous GRR is relative to the globally homogeneous GRR at capturing the E-AB dynamics for a particular parameter set, we need to develop a metric. 
We have a sequence of points, $(1,U_1),(2,U_2),...,(M,U_M)$, from the simulation, where $U_t$, as previously defined, is the number of agents in state $\U$ at iteration $t$ for $t=1,\ldots,M$.  
By linear spline interpolation of these points we will construct a function $g(t)$.  
We also have a sequence of points, $(1, \hat{U}_1),(2, \hat{U}_2),\dots,(M,\hat{U}_M)$, from the GRR.  
Given the fact that some of the error is due to translation and that the number of agents is much larger than the number of iterations, we need to normalize the data.  
We scale the $t$-values so that $t_i\leftarrow t_i/M$ and $\hat{t}_i \leftarrow \hat{t}_i/M$.  
Additionally, we let $\gamma = \max\{U_1,U_2,...,U_M\}$ and scale the $U$-values so that $U_i\leftarrow U_i/\gamma$ and $\hat{U}_i \leftarrow \hat{U}_i/\gamma$.  
Our error metric $\nu$ in Equation \eqref{metric2} is a normalized least square, evaluating the average distance of each scaled GRR estimation to the scaled E-AB simulated curve.  Details of the derivation for the error metric can be found in Appendix \ref{MetricDeriv}. 

\begin{table}[H]
\begin{center}
\caption{Error between GRR and 1000 E-AB simulations using metric $\nu$ from Equation \eqref{metric2} for infectivity radii $\rho_0$, contact tolerances $\kappa$, time in infected state $T_I$, and time in recovered state $T_R$.}
{\tiny
\subfloat[Error for number of infected agents.]{
  \begin{tabular}{|c||c|c|c|c||c|c|c|c|} \hline
  	\multirow{2}{*}{$N=10000$} 
		& \multicolumn{4}{c||}{$T_I=T_R=30$} 
		& \multicolumn{4}{c|}{$T_I=30$, $T_R=45$} \\            
	\cline{2-9}
  		& \multicolumn{2}{c|}{$\kappa=0.6$}
  		& \multicolumn{2}{c||}{$\kappa=0.8$}
  		& \multicolumn{2}{c|}{$\kappa=0.6$}
  		& \multicolumn{2}{c|}{$\kappa=0.8$} \\
  	\cline{2-9}
  	& Global & Local & Global & Local & Global & Local & Global & Local \\
  	\hline
		$\rho_0=0.02$ & 0.036606 & 0.009573 & 0.052660 & 0.020494 & 
			0.035438 & 0.010552 & 0.059061 & 0.021560 \\      
		\hline
		$\rho_0=0.04$ & 0.008092 & 0.001398 & 0.007514 & 0.001237 & 
			0.008595 & 0.001687 & 0.008653 & 0.001444 \\      
		\hline
		$\rho_0=0.08$ & 0.004019 & 0.000652 & 0.003252 & 0.000360 & 
			0.004686 & 0.000692 & 0.003657 & 0.000391 \\      
		\hline
		$\rho_0=0.16$ & 0.002030 & 0.000798 & 0.001589 & 0.000545 & 
			0.002352 & 0.000859 & 0.001827 & 0.000608 \\      
		\hline
  \end{tabular} \label{dataTableI}} \\
\subfloat[Error for number of recovered agents.]{
\begin{tabular}{|c||c|c|c|c||c|c|c|c|}
	\hline
	\multirow{2}{*}{$N=10000$} 
		& \multicolumn{4}{c||}{$T_I=T_R=30$} 
		& \multicolumn{4}{c|}{$T_I=30$, $T_R=45$} \\            
	\cline{2-9}
  		& \multicolumn{2}{c|}{$\kappa=0.6$}
  		& \multicolumn{2}{c||}{$\kappa=0.8$}
  		& \multicolumn{2}{c|}{$\kappa=0.6$}
  		& \multicolumn{2}{c|}{$\kappa=0.8$} \\
  	\cline{2-9}
  	& Global & Local & Global & Local & Global & Local & Global & Local \\
  	\hline
	$\rho_0=0.02$ & 0.022055 & 0.009537 & 0.027853 & 0.015889 & 
		0.028273 & 0.012485 & 0.036781 & 0.021142 \\      
	\hline
  	$\rho_0=0.04$ & 0.008436 & 0.000831 & 0.008046 & 0.000817 & 
  		0.010412 & 0.000895 & 0.010120 & 0.000937 \\      
	\hline
	$\rho_0=0.08$ & 0.003581 & 0.000126 & 0.003084 & 0.000104 & 
		0.004247 & 0.000138 & 0.003694 & 0.000103 \\      
	\hline
	$\rho_0=0.16$ & 0.001247 & 0.000176 & 0.001117 & 0.000154 & 
		0.001350 & 0.000189 & 0.001260 & 0.000184 \\      
	\hline
\end{tabular} \label{dataTableR} } 
} 
\end{center}
\end{table}

We see from Figure \ref{SIRCA_Simulations}, as well as Table \ref{dataTableI} and Table \ref{dataTableR}, that the locally homogeneous GRR approaches the E-AB with less error than the Global GRR.  
Despite the scaling and translation differences, the general behavioral trends of the Global GRR and locally homogeneous GRR emulate the E-AB agent state densities.

The surface plot in Figure \ref{fig: ErrorSurf} shows the mean error between the locally homogeneous GRR and the E-AB simulations with respect to the number of infected individuals. 
The horizontal axis represents  the expected number of susceptible agents in the initial infected agents neighborhood and the vertical axis represents the contact tolerance $\kappa$ for the mean error calculated from \eqref{metric2} using 150 iterations of data.  
We fixed the number of agents, $N$, and varied the infectivity radius, $\rho_0$, to generate the error surface plot in Figure \ref{fig: ErrorSurf}; however, one can generate similar error surface plots by fixing $N$ and varying $\rho_0$.  
 
Regardless of the contact tolerance $\kappa$, the GRR approaches the mean simulation's fixed point with only two
expected susceptible agents in the initial infected agent's neighborhood---the lower left subplot \textit{does} approach the fixed point when the simulation runs for more iterations.  However, the early infection front is better captured as the density of agents in the initial infectivity radius increases, as shown in the dark blue bands in the surface plot in Figure \ref{fig: ErrorSurf}.
In the yellow and orange bands of the surface plot  (approximately fewer than four susceptible agents in the initial infectivity radius), we find that in some simulations the infected agents transition to recovered before any susceptible agents become infected, reaching a point early in the epidemic where there are no infected agents.  These cases, in which the epidemic ``dies out,'' skews the expected number of infected, leading to higher error.  
The number of iterations required to match depends inversely on the expected number of susceptible agents in the initial infectivity radius.  For very low density simulations, where fewer than two susceptible agents are expected to be in the initial infectivity radius, the epidemic has a greater chance of ``dying out,'' which makes our current analysis unreliable.

\begin{figure}
\centering
\includegraphics[width=4in]{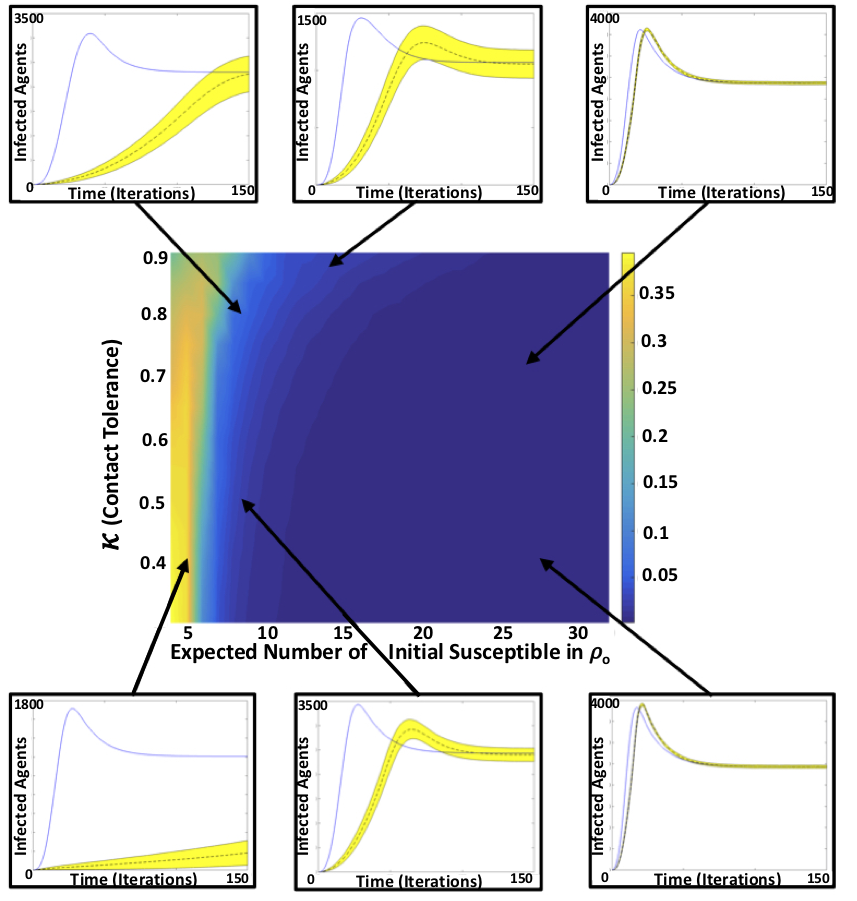}
\caption{The surface plot in the center displays the error between the locally homogeneous GRR and the E-AB simulations with respect to the number of infected agents as a function of the contact tolerance and the expected number of susceptible agents located in the initial infected agent's neighborhood at $t=0$.  The 6 outside plots show the locally homogeneous GRR infected population solution (blue) and the simulated solution (black) with a bound of $\pm$ one standard deviation (yellow).}
\label{fig: ErrorSurf}
\end{figure}

\section{Discussion and conclusions}
In this paper we have introduced a Global Recurrence Rule (GRR) for estimating the state densities for each iteration of a Markovian off-lattice AB model.  We demonstrated its utility with a three state Epidemiological Agent-Based (E-AB) model.  {However, this analysis can be used for any multi-state model of moving individuals that influence the states of other individuals within a neighborhood.} For the E-AB,
 we were able to perform stability analysis using the GRR, as well as determine bounds for the efficacy of the GRR in early stages of the epidemic.  
We note that other analytical techniques, besides stability analysis, can be explored with the GRR, including parameter sensitivity analysis and parameter estimation.  Performing these calculations directly with an AB model would prove to be very computationally expensive.  But, with the computational efficiency of the GRR, one could take existing epidemic data and find the E-AB parameters that best-fit the data. It is well known that the ease of quick and simple calibration of an AB model is critical \cite{Prieto2016}, and this methodology provides a framework to easily handle parameter sweeps.

We identified two classes of E-AB GRR: a globally homogeneous GRR, which assumes that the infected agents are uniformly distributed throughout the domain, and a locally homogeneous GRR, which assumes that there is an infectivity front expanding outward from the initially infected agent.  
With relaxed assumptions, the locally homogeneous GRR performs better than the globally homogeneous GRR with respect to early epidemic prediction.  However, we demonstrated and proved that the much simplified globally homogeneous GRR can predict long-term behavior just as well as the locally homogeneous GRR.

Further, we demonstrated that the GRR is a generalized model, but is not unique in its application---certain choices must be made.  
The generalized GRR definition lends itself to be used as a framework when adapting similar models.  For example, if the E-AB were three-dimensional or if the neighborhoods were a different geometry, then we could use our previously derived GRR equations \ref{Global GRRI}, \ref{Global GRRR}, and \ref{LocalGRRR} while only simply having to derive new expressions for $\mu(\N)$ and $\zeta_t$.  Further, we assumed a constant number of agents, $N$, but we could derive a GRR to calculate $S_t$, $I_t$, and $R_t$ that incorporates a dynamically varying number of agents in much the same way as we did in Section \ref{eca-results}.  The analysis would be similar, only in three-dimensional phase space instead of two-dimensional.

Previous analytical techniques, such as mean-field game theory, assumed the density of agents approaches infinity in order to calculate end behavior \cite{Lasry}.  Other approaches take continuum limits to approximate the dynamics of AB state distributions as a system of PDEs \cite{Chaturapruek,Devitt-Lee,Othmer}, which often corresponds to reducing the scales to infinitesimal time or spatial steps.  In contrast, the GRR analysis allows for and takes into account a finite number of agents in a discrete spatial and temporal domain, which in some cases might more closely reflect the outcome of interest for a particular application.  
Moreover, the GRR incorporates the notion that the state changes are incurred through spatially defined neighborhoods.  Traditional differential equation formulations of SIR models do not incorporate this feature since there is an assumption of a well-mixed population, but we saw that these neighborhoods affected our steady state solutions. Since our GRR analysis incorporates movements of individuals, this causes the contacts between individuals to be dynamic. We note that it is not feasible to do a direct comparison with a differential equation SIR model but that previous CA models with individuals at fixed locations and fixed neighborhoods have done some comparisons for specific cases \cite{Fuentes99}.

Our explicit GRR formulation for the E-AB model ultimately fails when the density of the infected population is zero.  In general, the expansion of the wave of infectivity is caused by the infection spread, rather than the agent movement.   
However, as can be seen in Figure \ref{fig: ErrorSurf}, when agent density is low, the early infectivity front growth relies on agent movement.
For the infection to not ``die out'' in these 
cases, we require an increase in the ratio of the movement size to the neighborhood area to increase the probability that a susceptible agent encounters the infectivity region. {Recent work has studied disease dynamics on dynamic networks \cite{Bansal10,Danon11,Enright18,Rocha16}; the threshold at which a disease can become an epidemic is dependent on reshaping of the contact network \cite{Schwarzkopf10,Segbroeck10}. Volz and Meyers also derived analytical results that have shown epidemics are possible for a range of a mixing parameter that controls the heterogeneity of contacts \cite{Volz09}.  In the future, it will be interesting to understand epidemic thresholds and continuum approximations of state changes in order to determine the probability that an infection will ``die out" using the GRR analysis for different movement rules for the agents, which could extensively change mixing and contacts of individuals.} This future analysis will establish density and parameter bounds for when the E-AB GRR formulation is reliable. Early predictions of disease dynamics are necessary  \cite{Pellis2015,Roberts2015}, and the proposed framework can be extended to determine accuracy of these estimates for given parameter regimes.

\newpage
\appendix
%%%%%%%%%%%%%%%%%%%%%%%%%%%%%%%%%%%%%%%%%%%%%%%%%%%%%%%%%%%%%%
\section{Derivation of metric}\label{MetricDeriv}
Suppose we create a linear spline interpolant $f(t)$ from a sequence of points $(t_1,S_1),(t_2,S_2),...,(t_N,S_M)$.  Suppose we also have another sequence of points $(\hat{t}_1, \hat{S}_1),(\hat{t}_2, \hat{S}_2),...,(\hat{t}_M,\hat{S}_M)$.  Our error metric is a normalized least square,
\begin{equation}
\nu = \frac{1}{M}\sum_{k=1}^M \inf_t \sqrt{ (\hat{t}_k-t)^2 + ( \hat{S}_k-f(t) )^2 },
\label{metric2}
\end{equation}
where our spline interpolation \cite{Stoer} is
$f(t) = \frac{S_{k+1}-S_k}{t_{k+1}-t_k}(t-t_k) + S_k, 
\text{for }t_k \leq t \leq t_{k+1}$.

Consider a point $(\hat{t}, \hat{S})$.  
We want to find $d(\hat \x, f) = \inf_t \sqrt{ (\hat{t} -t)^2 + (\hat{S} - f(t))^2 }$.  
The line that intersects points $(t_k, S_k)$ and $(t_{k+1}, S_{k+1})$ is given by 
\begin{equation}
\ell_k:= S = \frac{S_{k+1}-S_k}{t_{k+1}-t_k}(t-t_k) + S_k.
\label{line1}
\end{equation}  
To find $X$ such that $d(\hat\x, \ell_k)$ is minimized we find the line $\hat \ell$ that intersects $\hat \x$ and $\ell_k$.  This line is given by
\begin{equation}
\hat \ell:= S= -\frac{t_{k+1} - t_k}{S_{k+1} - S_k}(t-\hat t) + \hat S.
\label{line2}
\end{equation}
Setting (\ref{line1}) and (\ref{line2}) equal and solving for $t$ we find that
\begin{equation}
X = \frac{(t_{k+1}-t_k)(S_{k+1}-S_k)}{(t_{k+1}-t_k)^2 + (S_{k+1}-S_k)^2}
	\left( \frac{t_{k+1}-t_k}{S_{k+1}-S_k}\hat t +
	\frac{S_{k+1}-S_k}{t_{k+1}-t_k}t_k + \hat{S} - S_k \right).
\label{x-intersect}
\end{equation}
It follows that $Y=\ell_k (X)$.  We then calculate
\begin{equation}
d_k = \begin{cases}
\sqrt{ (\hat{t}-t_k)^2 + (\hat{S}-S_{k})^2 } & \text{: if } X<t_k, \\
\sqrt{ (\hat{t}-t_{k+1})^2 + (\hat{S}-S_{k+1})^2 } & \text{: if } S_{k+1} < X, \\
\sqrt{ (X-S_k)^2 + (Y-S_{k+1})^2 }  & \text{: otherwise.}
\end{cases}.
\end{equation}
We then have that 
$d(\hat \x, f) = \inf_t \sqrt{ (\hat{t}-t )^2 + 
( \hat{S}-f(t) )^2 } 
= \min\{d_1,d_2,...,d_{N-1}\}$.
It follows that we calculate the error by %:
$\nu = \sum_{k=1}^M d\left(\hat \x_k, f\right)$.

%%%%%%%%%%%%%%%%%%%%%%%%%%%%%%%%%%%%%%%%%%%%%%%%%%%%%%%%%%%%%%
%%%%%%%%%%%%%%%%%%%%%%%%%%%%%%%%%%%%%%%%%%%%%%%%%%%%%%%%%%%%%%
%%%%%%%%%%%%%%%%%%%%%%%%%%%%%%%%%%%%%%%%%%%%%%%%%%%%%%%%%%%%%%
%%%%%%%%%%%%%%%%%%%%%%%%%%%%%%%%%%%%%%%%%%%%%%%%%%%%%%%%%%%%%%
\section{Sparse $\mu\left(\tB_t^{\S,\I}\right)$ formula}\label{Section:Sparse}
Let us assume, as we did when deriving the locally homogeneous region, that the infected agents are on the radial center of mass of the region $\tB_t^{\S,\I} \setminus \tB_{t-1}^{\S,\I}$, as shown in Figure \ref{InfectionOverviewInPaper}.  We will assume that there are $n$ newly infected agents that are uniformly distributed on the radial center of mass, a distance of $r$ from the initially infected agent.

%%%%%%%%%%%%%%%%%%%%%%%%%%%%%%%%%%%%%%%%%%%%%%%%%%%%%%%%%%%%%%
\subsection{Deriving $\zeta_{k+1}$}
We want to find the total area $\mu\left( \bigcup_{i=1}^n A_i \right)$, where $n$ is the expected number of infected agents in the region $\tB_t^{\S,\I} \setminus \tB_{t-1}^{\S,\I}$ and $A_i$ is the region, illustrated in Figure \ref{InfectionOverviewInPaper}, of the $i$th infected agent.  
For our expository purposes, we will assume that the newly infected agents only exist in the edge of the expanding wave of infected agents, leading to the simplified derivation, $n = I_t - I_{t-1}$.

By the inclusion-exclusion principle \cite{VanLint} we find that
$\mu\left( \bigcup_{i=1}^n A_i \right) =
\sum_{i=1}^n \mu(A_i) - \sum_{i=1}^n \mu\left(A_i \cap A_{i+1}\right), 
\quad \text{where }A_{n+1}=A_1$.
%\end{equation}
Note that $\mu(A_i)=\mu(A_j)$ and that $\mu(A_i \cap A_{i+1}) = \mu(A_j \cap A_{j+1})$ for all $i,j=1,2,...,n$.  We then have that
\begin{equation}
\mu\left( \bigcup_{i=1}^n A_i \right) =
n\big( \mu(A) - \mu(A_1 \cap A_2) \big).
\label{inclusionexclusion}
\end{equation}

First we will find $\mu(A)$, the region shown in Figure \ref{AOverview}.
\begin{figure}[H]
\begin{center}
\subfloat[]{
     \includegraphics[width=0.2\textwidth]{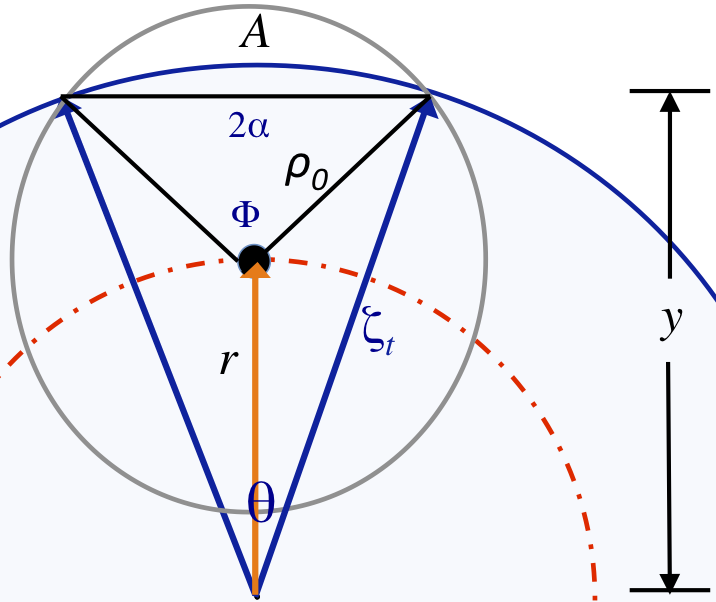} \label{AOverview}}	\qquad \qquad    
     \subfloat[]{
          \includegraphics[width=0.2\textwidth]{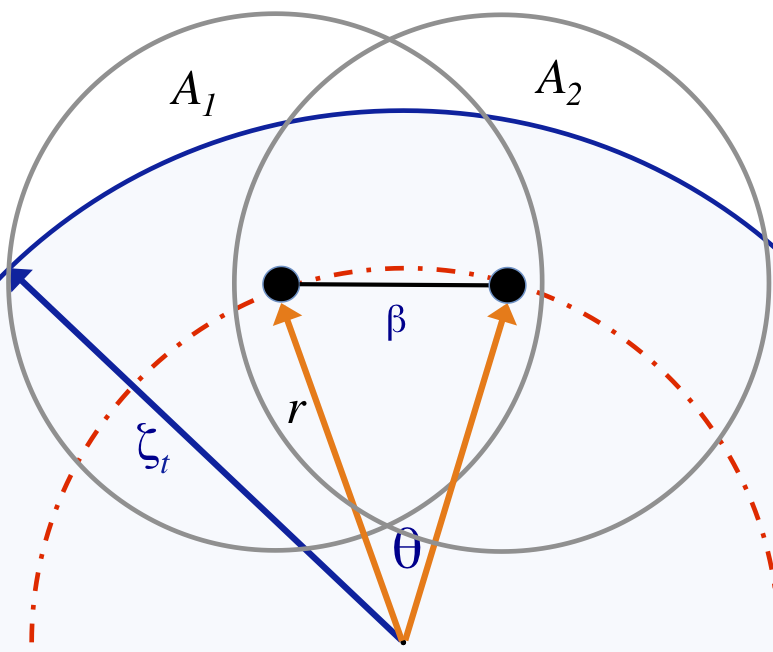} \label{A1A2Regions}	}	      
   \end{center}
  \caption{Newly infected agents (black circles) lie on the radial center of mass of the region $\tB_t^{\S,\I}\setminus \tB_{t-1}^{\S,\I}$, a distance $r$ from the initially infected agent.
  (\ref{AOverview}): Solving for $\mu(A)$ as part of calculating how large the infection front becomes.  The infectivity radius of the agent intersects the edge of region $\tB_t^{\S,\I}$ at two points, creating an angle $\theta$ from the center of the region and an angle $\Phi$ from the agent. 
  (\ref{A1A2Regions}): Solving for $\mu(A_1)\cap \mu(A_2)$.   The infected agents are a distance $\beta$ apart and form an angle $\theta$ from the center of the region $\tB_t^{\S,\I}$.
    }
\end{figure}
We already know $r$ and $\zeta_t$.  By our assumption, $\theta = 2\pi/n$,  
we will find $y$ by finding the intersection of $C_1$ and $C_2$ defined by
\begin{align*}
C_1 &: x^2 + (y-r)^2 = \rho_0^2, \\
C_2 &: x^2 + y^2 = \zeta_t^2.
\end{align*}
It follows that $y=\frac{\zeta_t^2 - \rho_0^2 + r^2}{2r}$.  
We can then find $\alpha = \sqrt{\zeta_t^2 - y^2}$ and $\Phi = 2\arcsin \left( \frac{\alpha}{\rho_0} \right)$.

From Figure \ref{ARegions} we know that $\mu(A) = \mu(R_2)-\mu(R_4)$.  It is clear that
$\mu(R_2) = \mu(R_1 \cup R_2) - \mu(R_1) 
= \frac{\Phi}{2}\rho_0^2 - \alpha \sqrt{\rho_0^2 - \alpha^2}$
and
$\mu(R_4) = \mu(R_3 \cup R_4) - \mu(R_3)
= \frac{\theta}{2}\zeta_t^2 - \alpha y$.

\begin{figure}[H]
\begin{center}
     \includegraphics[width=0.15\textwidth]{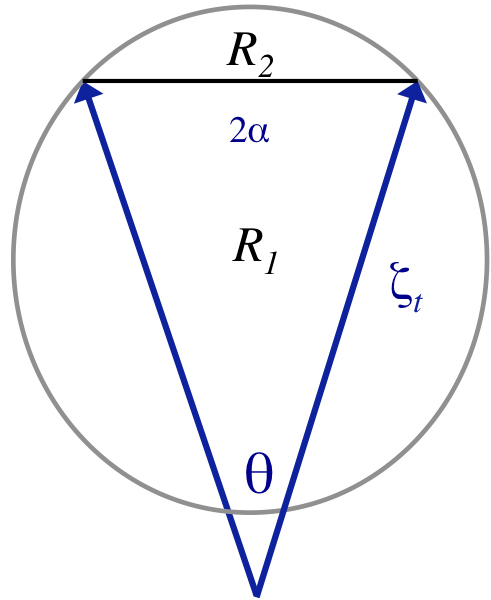}\qquad  \qquad
     \includegraphics[width=0.35\textwidth]{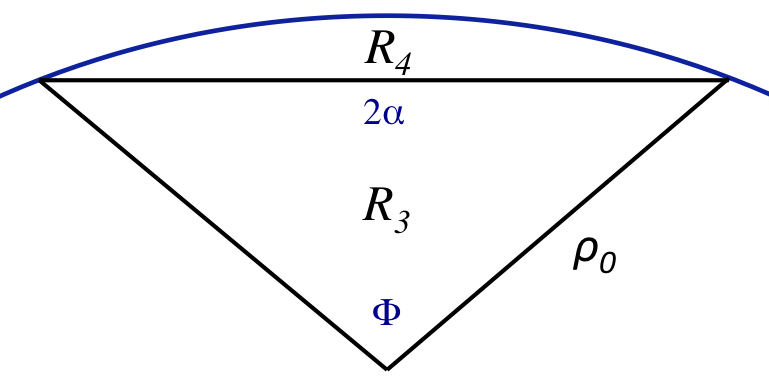}		      
   \end{center}
  \caption{We solve for region $A$ by subtracting $\mu(R_4)$ from $\mu(R_2)$.  We decompose solving for $\mu(A)$ in Figure \ref{AOverview} by solving for the outer sector (left) and the inner sector (right).}
  \label{ARegions}
\end{figure}

We then have that
\begin{equation}
\mu(A) = \frac{1}{2}\left(\Phi \rho_0^2 - \theta \zeta_t^2\right) 
- \alpha\left(\sqrt{\rho_0^2 - \alpha^2} - y\right).
\label{muA}
\end{equation}

Now we will find $\mu(A_1)\cap \mu(A_2)$.  From Figure \ref{A1A2Regions} we know $\theta$, $r$, and $\zeta_t$.  
We want to find the $x$-coordinate of the intersection of $C_1$ and $C_2$ defined as
\begin{align*}
C_1 &: (x+h)^2 + (y-k)^2 = \rho_0^2, \\
C_2 &: x^2 + y^2 = \zeta_t^2
\end{align*}
with $h=\beta/2$ and $k = \sqrt{r^2 - (\beta/2)^2}$.
The intersection is the larger solution $\hat x$ of the quadratic
$4\left(h^2+k^2\right)\hat{x}^2 + 
4h\left(2k^2+\eta\right)\hat{x} 
+ \Big(\eta^2 - 4k^2\left(\rho_0^2-h^2\right)\Big)=0$,
where $\eta = h^2-k^2+\zeta_k^2-\rho_0^2$.

We then have that
$\mu(A_1 \cap A_2) = 
2\int_0^{\hat{x}} \left( k + \sqrt{\rho_0^2 - (x+h)^2} - \sqrt{\zeta_t^2 - x^2}\right)\,dx$ if $\hat x>0$.
After integrating, if $\hat{x}>0$ we have
\begin{footnotesize}
\begin{equation}
\begin{split}
\mu(A_1\cap A_2) &= 
 (h+\hat{x})\sqrt{\rho_0^2 - (h+\hat{x})^2} 
+ \rho_0^2 \arctan\left(\frac{h+\hat{x}}{\sqrt{\rho_0^2 - (h+\hat{x})^2}}\right) 
- \zeta_t^2\arctan\left( \frac{\hat{x}}{\sqrt{\zeta_t^2-\hat{x}^2}} \right)
\\
& \quad -\hat{x}\sqrt{\zeta_t^2 - \hat{x}^2} + 
2k\hat{x} 
-\left[ h\sqrt{\rho_0^2 - h^2} 
+ \rho_0\arctan\left( \frac{h}{\sqrt{\rho_0^2-h^2}} \right) \right].
\end{split} 
\label{muA1A2}
\end{equation}
\end{footnotesize}

After inserting equations (\ref{muA}) and (\ref{muA1A2}) into equation (\ref{inclusionexclusion}), we have a computable formula for $\mu(\cup_{k=1}^n A_k)$.  
Our new wavefront radius is
\begin{equation} 
\zeta_{t+1} = \sqrt{ \frac{\pi \zeta_t^2 + \mu(\cup_{i=1}^n A_i)}{\pi} }.
\end{equation}

The above formulation works well for low density E-AB simulations, where agents in state $\R$ do not return to state $\S$ 
($T_R$ is longer than the time of the simulation).  
However, if recovered agents can become susceptible, then we must reformulate our calculation of the expected value of $n$.

\bibliographystyle{plain}
\bibliography{ms}

\end{document}